\documentclass[final,leqno,onefignum,onetabnum]{siamltex1213}

\usepackage{amsmath}
\usepackage{amssymb}
\usepackage{amsfonts}
\usepackage{epsfig}
\usepackage{boxedminipage}
\usepackage{multirow}
\usepackage{booktabs}
\usepackage{graphicx}
\usepackage{color}
\usepackage{epsfig,epsf,fancybox}
\usepackage{mathrsfs}
\usepackage{xspace}

\renewcommand{\theequation}{\thesection.\arabic{equation}}

\newcommand{\frc}[2]{{}^{#1\!\!}/_{\!#2}}
\newcommand{\rfrc}[2]{\raisebox{-.2ex}{$\frc{#1}{#2}$}}

\newcommand{\Bcal}{\mathcal{B}}

\newcommand{\Om}{{\mathcal{E}}}
\newcommand{\NOm}{{\overline{\mathcal{E}}}}

\newcommand{\BJ}{{\bf J}}
\newcommand{\BP}{{\bf P}}

\newcommand{\BW}{{\bf W}}
\newcommand{\BZ}{{\bf Z}}

\newcommand{\R}{\mathbb{R}}

\newcommand{\T}{\textnormal{T}}

\newcommand{\Diag}{\textsc{Diag}\xspace} 

\newcommand\doilink[1]{\href{http://dx.doi.org/#1}{#1}}
\newcommand\doi[1]{DOI: \doilink{#1}}
\newcommand{\Ma}{{\bf {M_{aug}}}}
\newcommand{\wa}{{\bf {w_{aug}}}}
\newcommand{\Na}{{\bf {N_{aug}}}}
\newcommand{\wak}{{\bf w}^{[k]}_{\bf aug}}
\newcommand{\fwas}{\widetilde{\bf w}^{*}_{\bf aug}}
\newcommand{\Nak}{{\bf N}^{[k]}_{\bf aug}}

\newcommand{\Rak}{{\bf R}^{[k]}_{\bf aug}}
\newcommand{\Ra}{{\bf R_{aug}}}
\newcommand{\be}{\begin{equation}}
\newcommand{\ee}{\end{equation}}
\newcommand{\ba}{\begin{array}}
\newcommand{\ea}{\end{array}}
\newcommand{\bpm}{\begin{pmatrix}}
\newcommand{\epm}{\end{pmatrix}}
\newcommand{\Dcal}{{\mathcal{D}}}
\newcommand{\Bx}{{\bf x}}
\newcommand{\Bv}{{\bf v}}
\newcommand{\Bz}{{\bf z}}
\newcommand{\Bw}{{\bf w}}

\newcommand{\Bb}{{\bf b}}
\newcommand{\Bd}{{\bf d}}
\newcommand{\Bh}{{\bf h}}
\newcommand{\By}{{\bf y}}
\newcommand{\Br}{{\bf r}}
\newcommand{\Bp}{{\bf p}}
\newcommand{\Ba}{{\bf a}}
\newcommand{\Bq}{{\bf q}}
\newcommand{\argmin}{\mathop{\rm argmin}}

\newcommand{\Bnu}{{\boldsymbol{\nu}}}

\newcommand{\barN}{{\Delta\!N}}
\renewcommand{\thefootnote}{\fnsymbol{footnote}}
\begin{document}

\title{Local Linear Convergence of ISTA and FISTA on the LASSO Problem \thanks{This research was partially supported by NSF grant IIS-1319749.}}
\date{}
\author{Shaozhe Tao\footnotemark[2]
\and Daniel Boley\footnotemark[3]
\and Shuzhong Zhang\footnotemark[4]
}
\maketitle

\renewcommand{\thefootnote}{\fnsymbol{footnote}}

\footnotetext[2]{Department of Industrial and Systems Engineering, University of Minnesota, Minneapolis, MN 55455, United States. Email: taoxx120@umn.edu.}
\footnotetext[3]{Department of Computer Science and Engineering, University of Minnesota, Minneapolis, MN 55455, United States. Email: boley@cs.umn.edu.}
\footnotetext[4]{Department of Industrial and Systems Engineering, University of Minnesota, Minneapolis, MN 55455, United States. Email: zhangs@umn.edu.}

\begin{abstract}
We establish local linear convergence bounds for the ISTA and FISTA iterations on the model
LASSO problem.  We show that FISTA can be viewed as an accelerated ISTA process.
Using a spectral analysis, we show that, when close enough to the solution, both iterations
converge linearly, but FISTA slows down compared to ISTA, making it advantageous to switch
to ISTA toward the end of the iteration processs.  We illustrate the results
with some synthetic numerical examples.
\end{abstract}

\section{Introduction}
The $l_1$-norm regularized least squares model has received much attention recently due to its wide applications in the real problems including compressed sensing \cite{CDS99}, statistics \cite{EHJ04}, sparse coding \cite{GL10}, geophysics \cite{TBM79} and so on. The problem in question is:
\be \label{LASSO}
\min_{x\in \R^{n}} \rfrc{1}{2}\|A\Bx - \Bb\|_2^2 + \lambda \|\Bx\|_1
\ee
where $A\in \R^{m\times n}$ is a given matrix, $\Bb$ is a given vector and $\lambda$ is a positive scalar.

The idea of $l_1$ regularization is decades old, but the least squares problem with $l_1$ penalty was presented and popularized independently under names Least Absolute Selection and Shrinkage Operator (LASSO) \cite{T94} and Basis Pursuit Denoising \cite{CDS99}. For example, in compressed sensing, we are interested in recovering a solution $\Bx$ to an undetermined system of linear equations $A\Bx = \Bb$ in the case where $n \gg m$. The linear algebra tells us that this linear system either does not exist or is not unique when the number of unknowns is greater than the number of equations. The conventional way to solve the system is to find the minimum $l_2$-norm solution, also known as linear least squares. However, if $\Bx$ is sparse, as very common in many applications, then $\Bx$ can be exactly recovered by computing the above $l_1$-norm regularized least-squares model. Since LASSO becomes the dominant expression describing this model, we will use term LASSO to denote the above model for the remainder of the paper.

Although the LASSO problem can be cast as a second order cone programming and solved by standard general algorithms like an interior point method \cite{BN01}, the computational complexity of such traditional methods is too high to handle large-scale data encountered in many real applications. Recently, a number of algorithms that take advantage of the special structure of the LASSO problem has been proposed. Among them, two remarkable ones are iterative shrinkage thresholding algorithm (ISTA) and its accelerated version fast iterative shrinkage thresholding algorithm (FISTA).

ISTA is also known as the proximal gradient method and its computation only involves matrix and vector multiplication, which has great advantage over standard convex algorithms by avoiding a matrix factorization \cite{PB13}. Recently, Beck and Teboulle \cite{BT09} proposed an accelerated ISTA, named as FISTA, in which a specific relaxation parameter is chosen. A similar algorithm to FISTA was also previously developed
independently by Nesterov in \cite{N83}. Both two algorithms are designed for solving problems containing convex differentiable objectives combined with an $l_1$ regularization
term as the following problem:
$$
\text{min} \{f(\Bx) + g(\Bx): \Bx\in \R^n\}
$$
where $f$ is a smooth convex function and $g$ is a continuous function but possibly nonsmooth.
Clearly, the LASSO problem is a special case of above formulation with $f(\Bx) = \rfrc{1}{2}\|A\Bx - \Bb\|^2$, $g(\Bx) = \|\Bx\|_1$. 
Its gradient  $\nabla f = A^{\T}A\Bx - A^{\T}\Bb$ is Lipschitz continuous with constant $L(f) = 2\rho(A^{\T}A)=\|A^{\T}A\|_2$, i.e., $\|\nabla f (\Bx_1) - \nabla f (\Bx_2) \| \leq L(f)  \|\Bx_1- \Bx_2\|,  \forall \Bx_1, \Bx_2 \in \R^n$.
It has been shown \cite{BT09} that FISTA provides a convergence rate of $O(1/k^2)$ compared to the rate of $O(1/k)$ by ISTA, while maintaining practically the same per-iteration cost,
where $k$ is the iteration number. However, in contrast to the results of a global convergence rate, as far as we know, there is no result on the local convergence behavior of standard ISTA and FISTA.

In this work, we establish local bounds on the convergence behavior of  ISTA and FISTA on the LASSO problem.
Comparing the two methods, we show how FISTA can be considered as an accelerated ISTA, but as one
approaches the solution, ISTA can actually be faster.
Extending the same techniques as in \cite{B13}, we show that linear convergence can be reached eventually, but not necessarily from the beginning.
Specifically, we give a way to represent the ISTA and FISTA iteration as a matrix recurrence and apply a spectral analysis on the corresponding eigenvalue problem.
We analyze the local behavior as it passes through several phases or ``regimes", treating each regime separately.
Based on the spectral analysis, each possible regime one can encounter during the course of the iteration is characterized.
Under normal circumstances, the theory predicts that either ISTA or FISTA should pass through several stages or ``regimes" of four different types, some of which consist of taking constant steps, but finally reaching a regime of linear convergence when close enough to the optimal solution.
Besides, with our analysis, more properties on FISTA and ISTA on the LASSO problem can be derived.

Throughout this paper, all vector and matrix norms are the $l_2$ norms (the largest singular value for a matrix) unless otherwise specified. For real symmetric matrices, the matrix $2$-norm is the same as the spectral radius (largest absolute value of any eigenvalue), hence we use those interchangeably for symmetric matrices.

The paper is organized as follows.
Section \ref{sec:prelim} gives some basic preliminaries of the paper.
We introduced how to formulate ISTA and FISTA into a matrix recurrence form in Sections \ref{sec:ISTArecur} \& \ref{sec:FISTArecur} and then derived spectral properties of the matrix operators in Sections \ref{sec:rak} \& \ref{sec:nak-prop}.
Section \ref{sec:regimes} gives details about four types of regimes that ISTA and FISTA will encounter in the iterations process based on our spectral analysis.
Our first main result is given in Section \ref{sec:linear}, which shows the local linear convergence of ISTA and FISTA on the LASSO problem.
In Section \ref{sec:accel} we compare the behavior in each regime, showing that FISTA can be faster that ISTA through most of the regimes, but
asymptotically can be slower as one approaches the optimal solution.
Section \ref{sec:examples} includes two numerical examples run by the standard ISTA and FISTA, to illustrate many of the predicted behaviors.

\section{Preliminaries} \label{sec:prelim}
\subsection{Optimality condition of the LASSO problem} \label{sec:LASSO}
The first order KKT optimality conditions for the LASSO problem
\eqref{LASSO}  are
\begin{equation}
 A^{\T}(\Bb-A\Bx)=\lambda\Bnu \label{KT-1}
\end{equation}
where each component of $\Bnu$ satisfies
\begin{equation}
\left\{\begin{matrix}
{\nu_i = \text{sign}(x_i)} & \text{if} ~ x_i \neq 0\\
-1 \le \nu_i \le +1 & \text{if} ~ x_i=0
\end{matrix} \right\} ~~~~ \label{KT-2}
~~~ \text{for}~~~ i = 1,2,\cdots.
\end{equation}
Here the ``sign'' function is defined as
$${
\textrm{sign}(x) = \begin{cases}
+1 & \textrm{if } x > 0 \\
 0 & \textrm{if } x = 0 \\
-1 & \textrm{if } x < 0. \\
\end{cases}
}$$

\subsection{Uniqueness} \label{sec:uniq}
There are various sufficient and necessary conditions for the uniqueness of the LASSO problem or its variants. For example, \cite{OPT99, CY09, F05} show different sufficient conditions and \cite{T13} studies the necessary conditions for the LASSO problem. In fact, the problem \eqref{LASSO} needs to have a unique solution in many situations. For example, in compressive sensing signal recovery, having non-uniqueness solutions will result in unreliable recovery given the data. We refer readers to \cite{T13, ZYC12} and references therein for the uniqueness of the LASSO problem.

\subsection{ISTA and FISTA iteration} \label{sec:def}
In this part, we review the basic iteration of ISTA and FISTA for solving the LASSO problem. To make clear the difference between ISTA and FISTA, we let $\widehat{\Bx}$ and $\widetilde{\Bx}$ denote the iterates of ISTA and FISTA respectively in the remainder of this paper. The basic step of ISTA for the LASSO problem can be reduced to \cite{DDD04, BT09}
\begin{eqnarray*}
\widehat{\Bx}^{[k+1]} & = & \argmin_{\widehat{\Bx}} \{g(\widehat{\Bx}) + \rfrc{L}{2}\|\widehat{\Bx}- (\widehat{\Bx}^{[k]}-\rfrc{1}{L}\nabla f(\widehat{\Bx}^{[k]}))\|^2\} \\
& = & \argmin_{\widehat{\Bx}} \{\lambda\|\widehat{\Bx}\|_1 + \rfrc{L}{2}\|\widehat{\Bx}- (\widehat{\Bx}^{[k]}-\rfrc{1}{L} ( A^{\T}A \widehat{\Bx}^{[k]} - A^{\T}\Bb ) \|^2\} \\
& = & \text{Shr}_{\rfrc{\lambda}{L}}((I - \rfrc{1}{L}A^{\T}A)\widehat{\Bx}^{[k]} + \rfrc{1}{L}A^{\T}\Bb).
\end{eqnarray*}

One advantage of ISTA is that the above step can be solved in closed form, leading to the following updates repeated until convergence, where $\widehat{\Bx}^{[k]}$ denote the vectors from previous pass, and $L$ is the given constant equal to $\|A^{\T}A\|_2$.
\begin{center}
\begin{tabular}{@{}llr@{}}\toprule
{Algorithm 1: One pass of ISTA}\\
\hline
start with $\widehat{\Bx}^{[k]}$. \\
\qquad Set $\widehat{\Bx}^{[k+1]}= \text{Shr}_{\rfrc{\lambda}{L}}((I - \rfrc{1}{L}A^{\T}A)\widehat{\Bx}^{[k]} + \rfrc{1}{L}A^{\T}\Bb)$.\\
Result is $\widehat{\Bx}^{[k+1]}$ for next pass.\\
\hline \label{Alg:ISTA}
\end{tabular}
\end{center}

As for FISTA, the difference from ISTA is that the shrinkage operator is not employed on the previous point $\widetilde{\Bx}^{[k-1]}$ but a point $\By^{[k]}$, which uses a very specific linear combination of the previous two points $\{\Bx^{[k-1]}, \widetilde{\Bx}^{[k-2]}\}$. The algorithm of FISTA for LASSO problem is presented as below, where the initial $\By^{[1]} = \widetilde{\Bx}^{[0]} \in \R^n$ and $t^{[0]}=t^{[1]}=1$.

\begin{center}
\begin{tabular}{@{}llr@{}}\toprule
{Algorithm 2: One pass of FISTA}\\
\hline
start with $t^{[k]}$, $t^{[k-1]}$, $\widetilde{\Bx}^{[k-1]}$ and $\widetilde{\Bx}^{[k-2]}$. \\
\qquad 1. Set $\By^{[k]}= \widetilde{\Bx}^{[k-1]}+ \frac{t^{[k-1]}-1}{t^{[k]}}(\widetilde{\Bx}^{[k-1]}-\widetilde{\Bx}^{[k-2]})  $. \\
\qquad 2. Set $\widetilde{\Bx}^{[k]}= \text{Shr}_{\rfrc{\lambda}{L}}((I - \rfrc{1}{L} A^{\T}A)\By^{[k]} + \rfrc{1}{L}A^{\T}\Bb)$.\\
\qquad 3. Set $t^{[k+1]} = \frac{1+\sqrt{1+4{t^{[k]}}^2}}{2}$. \\
Result is $t^{[k+1]}$, $t^{[k]}$, $\widetilde{\Bx}^{[k]}$, $\widetilde{\Bx}^{[k-1]}$ for next pass.\\
\hline
\end{tabular}
\end{center}

\section{Auxiliary Variables with Local Monotonic Behavior} \label{sec:aux}
\subsection{ISTA as a Matrix Recurrence} \label{sec:ISTArecur}
Instead of carrying the iteration using variables $\widehat{\Bx}^{[k]}$, we use two auxiliary variables to carry the iteration. One variable, namely, $\widehat{\Bw}^{[k]}$ exhibits smooth behavior, with linear convergence locally around a fixed point, and the other variable $\widehat{\Bd}^{[k]}$ is simply a ternary vector based on the three cases of the shrinkage operator. We let, for all $k$, the common iterate be
\be\label{ISTA-def-w}
\widehat{\Bw}^{[k]} = (I - \rfrc{1}{L}A^{\T}A)\widehat{\Bx}^{[k]} + \rfrc{1}{L}A^{\T}\Bb
\ee
and vector $\widehat{\Bd}^{[k]}$ be defined elementwise as
\begin{eqnarray}
\widehat{d}_i^{[k]} = \text{sign}(\text{Shr}_{\rfrc{\lambda}{L}}\widehat{w}_i^{[k]})) = &\begin{cases}
{1} & \text{if} ~ \widehat{w}_i^{[k]} > \rfrc{\lambda}{L} \\
0  & \text{if} ~ -\rfrc{\lambda}{L} \leq \widehat{w}_i^{[k]} \leq \rfrc{\lambda}{L}\\
{-1} & \text{if} ~ \widehat{w}_i^{[k]} < -\rfrc{\lambda}{L}. \\
\end{cases} \label{ISTA-def-D}
\end{eqnarray}

By the updating rule in Alg. 1 and above two equations, one can obtain the $\widehat{\Bx}^{[k]} $-update in terms of $\widehat{\Bw}^{[k]} $ and $\widehat{\Bd}^{[k]} $
\be\label{ISTA-x-w}
\widehat{\Bx}^{[k+1]} = \text{Shr}_{\rfrc{\lambda}{L}}(\widehat{\Bw}^{[k]}) = (\widehat{D}^{[k]})^2\widehat{\Bw}^{[k]} - \rfrc{\lambda}{L}\widehat{\Bd}^{[k]}
\ee
where the matrix $\widehat{D}^{[k]} = \text{diag}(\widehat{\Bd}^{[k]})$. Using \eqref{ISTA-def-w}, \eqref{ISTA-def-D} and \eqref{ISTA-x-w}, the update formula for $\widehat{\Bw}$ now can be expressed explicitly as follows:
\begin{eqnarray*}
\widehat{\Bw}^{[k+1]} &=& R^{[k]} \widehat{\Bw}^{[k]} + \widehat{\Bh}^{[k]}\\
 &=& [(I-\rfrc{1}{L}A^{\T}A)(\widehat{D}^{[k]})^2]\widehat{\Bw}^{[k]} - (I-\rfrc{1}{L}A^{\T}A) \rfrc{\lambda}{L}\widehat{\Bd}^{[k]} + \rfrc{1}{L}A^{\T}\Bb
\end{eqnarray*}
where we denote
\be
\begin{array}{ccl} \label{h-ista}
R^{[k]}&=&[(I-\rfrc{1}{L}A^{\T}A)(\widehat{D}^{[k]})^2] \\ 
 \widehat{\Bh}^{[k]} &=& -(I-\rfrc{1}{L}A^{\T}A) \rfrc{\lambda}{L}\widehat{\Bd}^{[k]} + \rfrc{1}{L}A^{\T}\Bb
\end{array}
\ee
throughout this paper. Therefore, the ISTA in Alg.\ 1 with variable $\widehat{\Bx}$ can be modified to the following procedure using the new variables $\widehat{\Bw}$ and $\widehat{D}$.
\begin{center}
\begin{tabular}{@{}llr@{}}\toprule
{Algorithm 3: One pass of modified ISTA}\\
\hline
start with $\widehat{\Bw}^{[k]}$,$\widehat{D}^{[k]}$.  \\
\qquad 1. $\widehat{\Bw}^{[k+1]}= R^{[k]} \widehat{\Bw}^{[k]} + \widehat{\Bh}^{[k]}$
(with $R^{[k]},\widehat{\Bh}^{[k]}$ defined by \eqref{h-ista}.\\
\qquad 2. $\widehat{D}^{[k+1]}=\textsc{{Diag}}(\textsc{sign}(\text{Shr}_{\rfrc{\lambda}{L}}(\widehat{\Bw}^{[k]})))$.\\
Result is $\widehat{\Bw}^{[k+1]}$, $\widehat{D}^{[k+1]}$ for next pass.\\
\hline \label{Alg:ISTA-2}
\end{tabular}
\end{center}
Alg.\ 3 is mathematically equivalent to Alg.\ 1 and is designed only for the purpose of analysis, not intended for computation. We note that step 1 of Alg.\ 3 can be written as a homogeneous matrix recurrence in \eqref{ISTA-Raug},
which we will use to characterize ISTA's convergence.
\be
\begin{aligned}\label{ISTA-Raug}
\bpm \widehat{\Bw}^{[k+1]} \\1 \epm &=&\Rak \bpm \widehat{\Bw}^{[k]} \\1 \epm = \bpm R^{[k]} & \widehat{\Bh}^{[k]} \\0&1 \epm \bpm \widehat{\Bw}^{[k]} \\1 \epm\\
&=& \bpm (I-\rfrc{1}{L}A^{\T}A)(\widehat{D}^{[k]})^2 & \widehat{\Bh}^{[k]} \\0&1 \epm \bpm \widehat{\Bw}^{[k]} \\1 \epm
\end{aligned}
\ee
where we denote $\Rak$ as $\bpm R^{[k]} & \widehat{\Bh}^{[k]} \\0&1 \epm$, the augmented matrix of $R^{[k]}$, in this paper.

{The following lemma shows the fixed point of Alg.\ 3 is a KKT point of the LASSO problem and vice versa.}
\begin{lemma} \label{ISTA-KT-2}
Suppose $\bpm \widehat{\Bw} \\1 \epm$ is an eigenvector corresponding to eigenvalue 1 of $\Ra$(omitting $[k]$) in \eqref{ISTA-Raug} and $\widehat{D} = \widehat{D}^{[k+1]} = \widehat{D}^{[k]}=\textsc{{Diag}}(\widehat{\Bd})$ with entries $\widehat{d}_i = 1$ if $\widehat{w}_i > \rfrc{\lambda}{L}$, $\widehat{d}_i = -1$ if $\widehat{w}_i < -\rfrc{\lambda}{L}$ and $\widehat{d}_i = 0$ if $-\rfrc{\lambda}{L} \leq \widehat{w}_i \leq \rfrc{\lambda}{L}$. Then the variable defined by $\widehat{\Bx}=\text{Shr}_{\rfrc{\lambda}{L}}(\widehat{\Bw})$ satisfies the $1$st order KKT condition. Conversely, if $\widehat{\Bx}$ and $\Bnu=\rfrc{1}{\lambda}A^{\T}(\Bb-A\widehat{\Bx})$ satisfy the KKT condition, then $\begin{pmatrix}\widehat{\Bw} \\1 \end{pmatrix}$, with $\widehat{\Bw} = \widehat{\Bx} + \rfrc{\lambda}{L}\Bnu$, is an eigenvector of $\Ra$ corresponding to eigenvalue 1, where $\Ra$ is defined as in \eqref{ISTA-Raug} and $\widehat{D}^{[k+1]} = \widehat{D}^{[k]} = \widehat{D} = \textsc{{Diag}}(\widehat{\Bd})$ with entries $\widehat{d}_i = 1$ if $\widehat{w}_i > \rfrc{\lambda}{L}$, $\widehat{d}_i = -1$ if $\widehat{w}_i < -\rfrc{\lambda}{L}$ and $\widehat{d}_i = 0$ if $-\rfrc{\lambda}{L} \leq \widehat{w}_i \leq \rfrc{\lambda}{L}$.
\end{lemma}

\begin{proof}
	From \eqref{ISTA-Raug}, we have
	\be \label{ISTA-kkt-proof1}
	(A^{\T}A\widehat{D}^2)\widehat{\Bw} + L(I-\widehat{D}^2)\widehat{\Bw} = \rfrc{\lambda}{L}A^{\T}A  \widehat{\Bd} +  A^{\T}\Bb -\lambda\widehat{\Bd}
	\ee
	and with the definition
	\be\label{ISTA-shrw}
	\widehat{x}_i=\text{Shr}_{\rfrc{\lambda}{L}}(\widehat{w}_i)=
	\left\{\begin{array}{lll}
		\widehat{w}_i - \rfrc{\lambda}{L} & \text{if} ~ \widehat{w}_i > \rfrc{\lambda}{L}
		& \Longleftrightarrow ~~ \widehat{d}_i = 1\\
		0 &  \text{if} ~ -\rfrc{\lambda}{L} \leq \widehat{w}_i \leq \rfrc{\lambda}{L}
		& \Longleftrightarrow ~~ \widehat{d}_i = 0\\
		\widehat{w}_i + \rfrc{\lambda}{L} & \text{if} ~ \widehat{w}_i < -\rfrc{\lambda}{L}
		& \Longleftrightarrow ~~ \widehat{d}_i = -1
	\end{array}\right.
	~~~~
	\ee
	we define a set $\Om$ s.t.
	$$
	\Om =  \left\{{i\in\left\{1,\cdots,n\right\}:|\widehat{d}_i| = 1}\right\} = \left\{{i\in\left\{1,\cdots,n\right\}:x_i \neq 0 }\right\}
	$$
	and denote $\NOm$ as the complement set of $\Om$. Then immediately by \eqref{ISTA-shrw}
	\be\label{ISTA-set}
	\widehat{\Bd}_{\NOm}=0, \quad \widehat{\Bx}_{\NOm}=0 , \quad
	\widehat{\Bw}_{\NOm} \in [-\rfrc{\lambda}{L},\rfrc{\lambda}{L}]
	\quad \text{and} \quad
	{(\widehat{\Bd}_{\Om})}_i =
	\begin{cases}
		1 & \text{if} ~ {(\widehat{\Bx}_{\Om})}_i > 0\\
		-1 & \text{if} ~ {(\widehat{\Bx}_{\Om})}_i <0
	\end{cases}.
	\ee
	We split matrix $A$ and do the permutation such that $A=[A_{\Om}, A_{\NOm}]$. So \eqref{ISTA-kkt-proof1} becomes
	\be\label{ISTA-kkt-proof2}
	\begin{split}
		\bpm A_{\Om}^{\T}A_{\Om} &0 \\ A_{\NOm}^{\T}A_{\Om}& 0 \epm \bpm \widehat{\Bw}_{\Om}\\ \widehat{\Bw}_{\NOm} \epm
		+ L\bpm 0_{\Om,\Om} & 0_{\Om,\NOm} \\ 0_{\NOm,\Om} & I_{\NOm,\NOm} \epm  \bpm \widehat{\Bw}_{\Om}\\ \widehat{\Bw}_{\NOm} \epm \\
		= \rfrc{\lambda}{L} \bpm A_{\Om}^{\T}A_{\Om} &   A_{\Om}^{\T}A_{\NOm} \\ A_{\NOm}^{\T}A_{\Om} &   A_{\NOm}^{\T}A_{\NOm} \epm \bpm \widehat{\Bd}_{\Om} \\ 0 \epm
		+ \bpm (A^{\T}\Bb)_{\Om} \\ (A^{\T}\Bb)_{\NOm} \epm - \lambda \bpm \widehat{\Bd}_{\Om} \\ 0 \epm.
	\end{split}
	\ee
	
	To satisfy the KKT condition, it is sufficient to show the following two equations
	\begin{eqnarray}
	\label{ISTA-lemma3.1a} A_{\Om}^{\T}(\Bb-A_{\Om}\widehat{\Bx}_{\Om}-A_{\NOm}\widehat{\Bx}_{\NOm})&=&\begin{cases}
	\lambda & \text{if} ~ {(\widehat{\Bx}_{\Om})}_i > 0\\
	-\lambda & \text{if} ~ {(\widehat{\Bx}_{\Om})}_i <0
	\end{cases} \\
	\label{ISTA-lemma3.1b}|A_{\NOm}^{\T}(\Bb-A_{\Om}\widehat{\Bx}_{\Om}-A_{\NOm}\widehat{\Bx}_{\NOm})| &\leq& \lambda.
	\end{eqnarray}
	To show \eqref{ISTA-lemma3.1a}, considering the case in $\Om$
	\begin{equation*}
	\begin{array}{rcl}
		\text{Equation} ~\eqref{ISTA-kkt-proof2}&\Longrightarrow & A_{\Om}^{\T}A_{\Om}\widehat{\Bw}_{\Om}
		= \rfrc{\lambda}{L}A_{\Om}^{\T}A_{\Om} \widehat{\Bd}_{\Om} + (A^{\T}\Bb)_{\Om} - \lambda \widehat{\Bd}_{\Om} \\
		&\Longrightarrow & A_{\Om}^{\T}A_{\Om}(\widehat{\Bw}_{\Om}-\rfrc{\lambda}{L} \widehat{\Bd}_{\Om})=(A^{\T}\Bb)_{\Om} - \lambda \widehat{\Bd}_{\Om} \\
		&\Longrightarrow & (A^{\T}\Bb)_{\Om}-A_{\Om}^{\T}A_{\Om}\widehat{\Bx}_{\Om}=\lambda \widehat{\Bd}_{\Om}.
	\end{array}
	\end{equation*}
	To show \eqref{ISTA-lemma3.1b}, considering the case in $\NOm$
	\begin{equation*}
	\begin{array}{rcl}
	\text{Equation} ~\eqref{ISTA-kkt-proof2} 
	 &\Longrightarrow & A_{\NOm}^{\T}A_{\Om}\widehat{\Bw}_{\Om}+L \widehat{\Bw}_{\NOm}
	= \rfrc{\lambda}{L}A_{\NOm}^{\T}A_{\Om} \widehat{\Bd}_{\Om} + (A^{\T}\Bb)_{\NOm}  \notag\\
	 &\Longrightarrow & A_{\NOm}^{\T}A_{\Om}(\widehat{\Bw}_{\Om}-\rfrc{\lambda}{L} \widehat{\Bd}_{\Om})+L \widehat{\Bw}_{\NOm}=(A^{\T}\Bb)_{\NOm} \notag\\
	 &\Longrightarrow & A_{\NOm}^{\T}A_{\Om}\widehat{\Bx}_{\Om}+L \widehat{\Bw}_{\NOm}=(A^{\T}\Bb)_{\NOm} \notag \label{ISTA-tmp1}\\
	 &\Longrightarrow & L \widehat{\Bw}_{\NOm}= (A^{\T}\Bb)_{\NOm}-A_{\NOm}^{\T}A_{\Om}\widehat{\Bx}_{\Om} \\
	 &\Longrightarrow&|A_{\NOm}^{\T}(\Bb-A_{\Om}\widehat{\Bx}_{\Om})|\leq \lambda. 
	\end{array}
	\end{equation*}

	Conversely, from KKT condition \eqref{KT-1}, we have $\rfrc{1}{L}A^{\T}\Bb - \rfrc{1}{L}A^{\T}A\widehat{\Bx} = \rfrc{\lambda}{L}\Bnu$ and then $\widehat{\Bx} = (I- \rfrc{1}{L}A^{\T}A) \widehat{\Bx} + \rfrc{1}{L}A^{\T}\Bb - \rfrc{\lambda}{L}\Bnu$. By KKT condition \eqref{KT-2}, we obtain \begin{eqnarray*}
		\widehat{\Bx} &=& \text{Shr}_{\rfrc{\lambda}{L}}((I - \rfrc{1}{L}A^{\T}A)\widehat{\Bx} + \rfrc{1}{L}A^{\T}\Bb) .
	\end{eqnarray*}
	This shows that KKT point of the LASSO can be written in the form of ISTA iterates. Then by one pass of our transformation, we get the matrix recurrence form as \eqref{ISTA-Raug}.
\qquad \end{proof}

\subsection{FISTA as a Matrix Recurrence} \label{sec:FISTArecur}
Similar to ISTA, we use auxiliary variables $\widetilde{\Bw}^{[k]}, \widetilde{D}^{[k]}$ to replace variable $\widetilde{\Bx}^{[k]}$ for carrying the FISTA iterations. We set
\be\label{FISTA-def-w}
\widetilde{\Bw}^{[k]} = (I - \rfrc{1}{L}A^{\T}A)\By^{[k]} + \rfrc{1}{L}A^{\T}\Bb.
\ee
Hence,
\be\label{FISTA-x-w}
\widetilde{\Bx}^{[k+1]} = \text{Shr}_{\rfrc{\lambda}{L}}(\widetilde{\Bw}^{[k]}) = (\widetilde{D}^{[k]})^2\widetilde{\Bw}^{[k]} - \rfrc{\lambda}{L}\widetilde{\Bd}^{[k]}.	
\ee
where for all $k$, the matrix $\widetilde{D}^{[k]} = \text{diag}(\widetilde{\Bd}^{[k]})$, and the vector $\widetilde{\Bd}^{[k]}$ is defined elementwise as
\begin{eqnarray}
\widetilde{d}_i^{[k]} = \text{sign}(\text{Shr}_{\rfrc{\lambda}{L}}\widetilde{w}_i^{[k]})) = &\begin{cases}
{1} & \text{if} ~ \widetilde{w}_i^{[k]} > \rfrc{\lambda}{L} \\
0  & \text{if} ~ -\rfrc{\lambda}{L} \leq \widetilde{w}_i^{[k]} \leq \rfrc{\lambda}{L}\\
{-1} & \text{if} ~ \widetilde{w}_i^{[k]} < -\rfrc{\lambda}{L}. \\
\end{cases} \label{FISTA-def-D}
\end{eqnarray}
Using \eqref{FISTA-def-w}, \eqref{FISTA-x-w} and the updating formula in Alg. 2, we arrive at
\be
\begin{array}{rcl}\label{eq-S-P-Q}
\widetilde{\Bw}^{[k+1]} &=& (I - \rfrc{1}{L}A^{\T}A)\left[\widetilde{\Bx}^{[k]}+\frac{t^{[k]}-1}{t^{[k+1]}} (\widetilde{\Bx}^{[k]}-\widetilde{\Bx}^{[k-1]})\right] + \rfrc{1}{L} A^{\T}\Bb \vspace{.5mm} \\ 
&=& (I- \rfrc{1}{L} A^{\T}A) 
\left[(\frac{t^{[k]}-1}{t^{[k+1]}}+1)((\widetilde{D}^{[k]})^2\widetilde{\Bw}^{[k]} - \rfrc{\lambda}{L}\widetilde{\Bd}^{[k]}) \right] \vspace{.5mm}\\
&& -(I- \rfrc{1}{L} A^{\T}A) 
\left[\frac{t^{[k]}-1}{t^{[k+1]}}( (\widetilde{D}^{[k-1]})^2\widetilde{\Bw}^{[k-1]} + \rfrc{\lambda}{L}\widetilde{\Bd}^{[k-1]})\right] 
+\rfrc{1}{L}A^{\T}\Bb \vspace{.5mm}\\
&=& (1+\frac{t^{[k]}-1}{t^{[k+1]}}) \left[(I- \rfrc{1}{L} A^{\T}A )(\widetilde{D}^{[k]})^2\right]\widetilde{\Bw}^{[k]} \vspace{.5mm}\\
&&- \frac{t^{[k]}-1}{t^{[k+1]}}\left[(I- \rfrc{1}{L} A^{\T}A )(\widetilde{D}^{[k-1]})^2\right]\widetilde{\Bw}^{[k-1]}  \vspace{.5mm}\\
&& + (I- \rfrc{1}{L} A^{\T}A ) \left[ - (1+\frac{t^{[k]}-1}{t^{[k+1]}})\rfrc{\lambda}{L}\widetilde{\Bd}^{[k]}+ \frac{t^{[k]}-1}{t^{[k+1]}} \rfrc{\lambda}{L}\widetilde{\Bd}^{[k-1]}\right]
 + \rfrc{1}{L} A^{\T}\Bb \vspace{.5mm}\\
&=& (1+\tau^{[k]})\widetilde{R}^{[k]}\widetilde{\Bw}^{[k]} - \tau^{[k]} \widetilde{R}^{[k-1]}\widetilde{\Bw}^{[k-1]} + \bar{\Bh} \vspace{.5mm}\\
& =&  P^{[k]}\widetilde{\Bw}^{[k]} + Q^{[k-1]}\widetilde{\Bw}^{[k-1]} + \bar{\Bh}^{[k]}
\end{array}
\ee
where we denote
\be
\begin{array}{ccl}\label{R-tilde}
	 \tau^{[k]} &=& \frac{t^{[k]}-1}{t^{[k+1]}}\\ 
	 P^{[k]} &=&  (1+\tau^{[k]}) \widetilde{R}^{[k]}\\
	Q^{[k]} &=& -\tau^{[k+1]} \widetilde{R}^{[k]} \\
	 \widetilde{R}^{[k]}  &=& (I- \rfrc{1}{L} A^{\T}A )(\widetilde{D}^{[k]})^2 \\ 
	 \bar{\Bh}^{[k]}  &=& (I- \rfrc{1}{L} A^{\T}A ) \left[ - (1+\tau^{[k]} )\rfrc{\lambda}{L}\widetilde{\Bd}^{[k]}+ \tau^{[k]}  \rfrc{\lambda}{L}\widetilde{\Bd}^{[k-1]}\right]
	 + \rfrc{1}{L} A^{\T}\Bb 
\end{array}
\ee
in the rest of this paper. Note that $R^{[k]}$ in \eqref{h-ista} refers to the mapping at the $k$-th iteration of ISTA while
$\widetilde{R}^{[k]}$ in \eqref{R-tilde} refers to the mapping that would occur if one took one step of ISTA starting at the $k$-th iterate of FISTA.
For the purposes of analysis, the modified FISTA iteration then can be equivalently expressed as in Alg.\ 4.
\begin{center}
\begin{tabular}{@{}llr@{}}\toprule
{Algorithm 4: One pass of modified FISTA}\\
\hline
start with $\widetilde{\Bw}^{[k-1]}$, $\widetilde{\Bw}^{[k]}$, $t^{[k]}$, $\widetilde{D}^{[k-1]}$ and $\widetilde{D}^{[k]}$. \\
\qquad 1. $\widetilde{\Bw}^{[k+1]}= P^{[k]}\widetilde{\Bw}^{[k]} + Q^{[k-1]}\widetilde{\Bw}^{[k-1]} + \bar{\Bh}^{[k]}$
(with $P^{[k]},Q^{[k-1]},\bar{\Bh}^{[k]}$ defined by \eqref{R-tilde}).\\
\qquad 2. $t^{[k+1]} = \frac{1+\sqrt{1+4{t^{[k]}}^2}}{2}$ {so that $\tau^{[k]}  = \frac{t^{[k]}-1}{t^{[k+1]}}$}.\\
\qquad 3. $\widetilde{D}^{[k+1]}=\textsc{{Diag}}(\textsc{sign}(\text{Shr}_{\rfrc{\lambda}{L}}(\widetilde{\Bw}^{[k]})))$.\\
Result is $\widetilde{\Bw}^{[k]}$, $\widetilde{\Bw}^{[k+1]}$, $t^{[k+1]}$, $\widetilde{D}^{[k]}$ and $\widetilde{D}^{[k+1]}$ for next pass.\\
\hline \label{Alg:FISTA-2}
\end{tabular}
\end{center}
Step 1 of above procedure can also be formulated as a homogeneous matrix recurrence analogous to \eqref{ISTA-Raug} for ISTA, but
with a larger (approximately double) dimension: 
\be \label{FISTA-Naug}
\begin{aligned}
&\bpm \widetilde{\Bw}^{[k+1]} \\ \widetilde{\Bw}^{[k]} \\ 1 \epm &=&\Nak \bpm \widetilde{\Bw}^{[k]} \\ \widetilde{\Bw}^{[k-1]} \\ 1 \epm
=\bpm N^{[k]} & \widetilde{\Bh}^{[k]}\\0&1 \epm \bpm \widetilde{\Bw}^{[k]} \\ \widetilde{\Bw}^{[k-1]} \\ 1 \epm\\
&&=& \bpm P^{[k]} & Q^{[k-1]} & \bar{\Bh}^{[k]} \\ I & 0 & 0 \\ 0 & 0 & 1 \epm
\bpm \widetilde{\Bw}^{[k]} \\ \widetilde{\Bw}^{[k-1]} \\ 1 \epm.
\end{aligned}
\ee
We denote
$
N^{[k]} =\bpm P^{[k]} & Q^{[k-1]} \\ I & 0 \epm ~~ \text{and} ~~ \widetilde{\Bh}^{[k]} = \bpm \bar{\Bh}^{[k]} \\0 \epm
$
such that
$\Nak = \bpm N^{[k]} & \widetilde{\Bh}^{[k]}\\0&1 \epm $ in the remainder of this paper.

{The following lemma shows the fixed point of Alg.\ 4 is a KKT point of the LASSO problem and vice versa,
analogous to Lemma \ref{ISTA-KT-2} for ISTA.}
\begin{lemma} \label{FISTA-KT-2}
	Suppose $\bpm \widetilde{\Bw}_1 \\\widetilde{\Bw}_2 \\1 \epm$ is an eigenvector corresponding to eigenvalue 1 of $\Na$ (omitting $[k]$)  in \eqref{FISTA-Naug},
	then $\widetilde{\Bw}_1 = \widetilde{\Bw}_2 := \widetilde{\Bw} $.  Suppose further that $\widetilde{D} = \widetilde{D}^{[k+1]} = \widetilde{D}^{[k]}=\textsc{{Diag}}(\widetilde{\Bd})$ with entries $\widetilde{d}_i = 1$ if $\widetilde{w}_i > \rfrc{\lambda}{L}$, $\widetilde{d}_i = -1$ if $\widetilde{w}_i < -\rfrc{\lambda}{L}$ and $\widetilde{d}_i = 0$ if $-\rfrc{\lambda}{L} \leq \widetilde{w}_i \leq \rfrc{\lambda}{L}$.
	Then the variables defined by $\widetilde{\Bx}=\text{Shr}_{\rfrc{\lambda}{L}}(\widetilde{\Bw})$ satisfies the $1$st order KKT conditions \eqref{KT-1} \eqref{KT-2}. Conversely, if $\widetilde{\Bx}$ and $\Bnu=\rfrc{1}{\lambda}A^{\T}(\Bb-A\widetilde{\Bx})$ satisfy the KKT conditions, then $\begin{pmatrix}\widetilde{\Bw} \\\widetilde{\Bw} \\1 \end{pmatrix}$, with $\widetilde{\Bw} = \widetilde{\Bx} + \rfrc{\lambda}{L}\Bnu$, is an eigenvector of $\Na$ corresponding to eigenvalue 1, where $\Na$ is defined as in \eqref{FISTA-Naug} and $\widetilde{D} = \widetilde{D}^{[k+1]} = \widetilde{D}^{[k]}=\textsc{{Diag}}(\widetilde{\Bd})$ with entries $\widetilde{d}_i = 1$ if $\widetilde{w}_i> \rfrc{\lambda}{L}$, $\widetilde{d}_i = -1$ if $\widetilde{w}_i < -\rfrc{\lambda}{L}$ and $\widetilde{d}_i = 0$ if $-\rfrc{\lambda}{L} \leq \widetilde{w}_i \leq \rfrc{\lambda}{L}$.
\end{lemma}

\begin{proof}
	The first statement is directly from the second block row of \eqref{FISTA-Naug}.
	With $\widetilde{D} = \widetilde{D}^{[k+1]} = \widetilde{D}^{[k]}$, a fixed
	point of \eqref{FISTA-Naug}
	becomes (omitting the superscripts $k$)
	\be
	\bpm \widetilde{\Bw} \\ \widetilde{\Bw} \\ 1 \epm = \bpm P & Q & \bar{\Bh} \\ I & 0 & 0 \\ 0 & 0 & 1 \epm
	\bpm \widetilde{\Bw} \\ \widetilde{\Bw} \\ 1 \epm
	\ee
yielding	
	\be \label{FISTA-kkt-proof1}
	(A^{\T}A\widetilde{D}^2)\widetilde{\Bw} + L(I-\widetilde{D}^2)\widetilde{\Bw} = \rfrc{\lambda}{L}A^{\T}A  \widetilde{\Bd} +  A^{\T}\Bb -\lambda\widetilde{\Bd}
	\ee
	which is exactly \eqref{ISTA-kkt-proof1}. So the rest part is the same as ISTA.
	
	Conversely, for the given KKT point $\widetilde{\Bx}$, If we define $\By = \widetilde{\Bx}+ \frac{t^{[k-1]}-1}{t^{[k]}}(\widetilde{\Bx}-\widetilde{\Bx})$ where $t^{[k]}$ is updated as $t^{[k+1]} = \frac{1+\sqrt{1+4{t^{[k]}}^2}}{2}$, one could show, just as in the proof of Theorem \ref{ISTA-KT-2}, that
	$\widetilde{\Bx} = \text{Shr}_{\rfrc{\lambda}{L}}((I - \rfrc{1}{L}A^{\T}A)\widetilde{\Bx} + \rfrc{1}{L}A^{\T}\Bb) = \text{Shr}_{\rfrc{\lambda}{L}}((I - \rfrc{1}{L}A^{\T}A)\By + \rfrc{1}{L}A^{\T}\Bb)$. This is exactly the FISTA iterates. By one pass of our  transformation, we obtain the matrix recurrence form as \eqref{FISTA-Naug}.
\qquad \end{proof}
For the preparation of the further discussion, we make three remarks.

(a). $\tau^{[k]} \longrightarrow 1$ from below as $k \longrightarrow \infty$.

(b). $\widetilde{R}^{[k]}=R^{[k]}$ if $\widehat{D}^{[k]} = \widetilde{D}^{[k]}$ and $\bar{\Bh}^{[k]} = \widehat{\Bh}^{[k]}$ if $\widehat{D}^{[k]} = \widetilde{D}^{[k]} = \widetilde{D}^{[k-1]}$.  This observation relates the $\Ra$ to $\Na$. It is the foundation upon which we establish the properties to compare ISTA and FISTA in Section \ref{sec:accel}.

(c). One main difference between operators of ISTA and FISTA (i.e. $\Ra$ and $\Na$) is that $\Rak$ is fixed when the flag matrix is fixed while $\Nak$ changes at each step $k$ even if the flag matrix is fixed. In other words, for all $k$, $\Rak= {\bf R}^{[k+1]}_{\bf aug}$ if $\widehat{D}^{[k]} = \widehat{D}^{[k+1]}$.  But $\Nak \neq {\bf N}^{[k+1]}_{\bf aug}$ even if $\widetilde{D}^{[k]} = \widetilde{D}^{[k+1]}$. The reason is that $\Nak$ depends on the changing stepsize $\tau^{[k]}$. Nevertheless, one can still use the same similar argument for $\Nak$ as for $\Rak$ by additional lemmas, as we will show in Section \ref{sec:regimes}. 

\subsection{Properties of $\Rak$} \label{sec:rak}
It is seen that $\Rak$ and $\Nak$ play key roles in the convergence. Hence we now focus on the spectral properties of $\Rak$ in this part and $\Nak$ in next part. Before that, we first recall some theory relating the spectral radius to the matrix norm. 
\begin{theorem}\label{thm:spectral property}Let $\rho(M)$ denote the spectral radius of an arbitrary square real matrix $M$, and let $\|M\|_2 $ denote the matrix 2-norm. Then we have the following:
	
	{\em (a)}.  $\rho(M)\leq\|M\|_2 $.

	{\em (b)}. If $\|M\|_2 = \rho(M)$ then for eigenvalue $\lambda$ such that $|\lambda|=\rho(M)$, the algebraic and geometric multiplicities of $\lambda$ are the same (all Jordan blocks for $\lambda$ is $1\times1$). Such a matrix is said to be a member of Class M \cite{H64,J72}.
	
	{\em (c)}. If a $\lambda$ such that $|\lambda|=\rho(M)$ has a Jordan block of dimension larger than $1$ (the geometric multiplicity is strictly less than the algebraic multiplicity),
	then for any $\epsilon > 0 $ there exists a matrix norm $\|\cdot\|_G$ (based on a nonsingular
	matrix $G$) such that $\rho(M) < \|M\|_G \leq \rho(M)+ \epsilon$.
\end{theorem}

We refer reader \cite{B13, J72, H64, GL13} for the proof of the above theorem.

\begin{lemma}\label{lemma:ISTA-norm}
Regarding ISTA, there are three properties of $R^{[k]}$:

{\em (a)}. $\|R^{[k]}\| = \|(I-\rfrc{1}{L}A^{\T}A)(\widehat{D}^{[k]})^2\| \leq 1$.

{\em (b)}. All eigenvalues must lie in the interval $[0,1]$.

{\em (c)}. If there exists one or more eigenvalues equal to $1$, then eigenvalue $1$  must have a complete set of eigenvectors.
\end{lemma}

\begin{proof}
We here omit the pass number $ ^{[k]}$ for simplicity.

(a). $\|R\| = \|(I-\rfrc{1}{L}A^{\T}A)\widehat{D}^2\| \leq \|(I-\rfrc{1}{L}A^{\T}A)\|\| \widehat{D}^2\| \leq 1$.

(b). The eigenvalue of $R$ is the same as the eigenvalue of $R'=\widehat{D}(I-\rfrc{1}{L}A^{\T}A)\widehat{D}$. Noticing $L = \|A^{\T}A\|_2$, we obtain $\|R'\|_2 \leq \|\widehat{D}^2\| \|(I-\rfrc{1}{L}A^{\T}A)\| \leq 1$. In addition, $R'$ is symmetric and a positive semidefinite matrix. Hence all eigenvalues of $R'$ must lie in the interval $[0,1]$. Hence should those of $R$.

(c). Because $\rho(R) = \|R\|_2 = 1$, this statement follows directly from Lemma \ref{thm:spectral property}.
\qquad \end{proof}

\subsection{Properties of $\Nak$} \label{sec:nak-prop}
Now we turn to the spectral analysis of the FISTA operator. Lemma \ref{lemma:FISTA-norm} demonstrates the eigenstructure of $\Nak$ and its relation to that of $\Rak$. 

{
	\begin{lemma}\label{lemma:FISTA-norm}
		We let $\gamma$ and $\beta$ denote the eigenvalue of $N^{[k]}$ and $\widetilde{R}^{[k]}$ respectively. 
		Any eigenvector $\bpm \widetilde{\Bw}_1 \\ \widetilde{\Bw}_2 \epm $ of $N^{[k]}$ corresponding to eigenvalue
		$\gamma$ must satisfy
		$\widetilde{\Bw}_1 = \gamma \widetilde{\Bw}_2$. Suppose $\widetilde{D}^{[k]}=\widetilde{D}^{[k-1]}$ and hence
		$\widetilde{R}^{[k]}=\widetilde{R}^{[k-1]}$, then (omitting index $^{[k]}$) we have the following results:
		
		{\em (a)}.  $\widetilde{\Bw}_2$ must also be an eigenvector of $\widetilde{R}$ with eigenvalue $\beta$, where $\beta$ and $\gamma$ has the relation $ \gamma^2 - \gamma(1+\tau)\beta + \tau\beta=0$.
		
		{\em (b)}. For $0 <\tau \leq 1$, the eigenvalues of $N$ defined in \eqref{FISTA-Naug} lie in the closed disk in the complex plane with center $\rfrc{1}{2}$ and radius $\rfrc{1}{2}$, denoted as $\Dcal(\rfrc{1}{2}, \rfrc{1}{2})$. In particular, if $N$ has any eigenvalue with absolute value $\rho(N)=1$, then that eigenvalue must be exactly 1.
		
		{\em (c)}. $N$ has an eigenvalue equal to $1$ if and only if $\widetilde{R}$ has an eigenvalue equal to $1$.
		
		{\em (d)}. Assuming $\tau<1$, then if $N^{[k]}$ has an eigenvalue equal to 1, this eigenvalue must have a complete set of eigenvectors.
	\end{lemma}
	
	\begin{proof}
		By the definition of $N^{[k]}$ (just after \eqref{FISTA-Naug})
		$$ N^{[k]}\cdot \bpm \widetilde{\Bw}_1 \\ \widetilde{\Bw}_2 \epm =\bpm P^{[k]} & Q^{[k-1]} \\ I & 0 \epm \bpm \widetilde{\Bw}_1 \\ \widetilde{\Bw}_2 \epm
		= \bpm P^{[k]}\widetilde{\Bw}_1+Q^{[k-1]}\widetilde{\Bw}_2 \\ \widetilde{\Bw}_1 \epm
		= \gamma\bpm \widetilde{\Bw}_1 \\ \widetilde{\Bw}_2 \epm$$
		and thus $\widetilde{\Bw}_1 = \gamma \widetilde{\Bw}_2$ is observed from the second row.
		
		(a). $$
		N\cdot \bpm \widetilde{\Bw}_1 \\ \widetilde{\Bw}_2 \epm  \\ = \bpm P & Q \\ I & 0 \epm \bpm \gamma \widetilde{\Bw}_2 \\ \widetilde{\Bw}_2 \epm = \bpm (1+\tau)\gamma \widetilde{R}\widetilde{\Bw}_2 -\tau \widetilde{R}\widetilde{\Bw}_2 \\\gamma \widetilde{\Bw}_2 \epm = \gamma \bpm \gamma \widetilde{\Bw}_2 \\ \widetilde{\Bw}_2 \epm
		$$
		and therefore,
		\be\label{FISTA-beta-gamma}
		\widetilde{R}\widetilde{\Bw}_2 = \frac{\gamma^2}{\left[(1+\tau) \gamma - \tau \right]} \widetilde{\Bw}_2 = \beta \widetilde{\Bw}_2 ~\Longrightarrow~ \gamma^2 - (1+\tau)\gamma\beta + \tau\beta=0.
		\ee

		(b).  We first study the spectrum of matrix $N- \rfrc{1}{2}I$, then the spectrum of $N$ should be a shift by $\rfrc{1}{2}I$. Let $\alpha = \gamma - \rfrc{1}{2}$ be the eigenvalue of $N- \rfrc{1}{2}I$ associated with eigenvector $\bpm \widetilde{\Bw}_1 \\ \widetilde{\Bw}_2 \epm $, then according to \eqref {FISTA-beta-gamma}, $\alpha$ and $\beta$ have the relation
		\begin{equation*}
		\alpha^2 + (1- \beta - \tau \beta) \alpha  + \rfrc{1}{2}\tau \beta - \rfrc{1}{2}\beta + \rfrc{1}{4} =0.
		\end{equation*}
		Note that $\tau \in (0,1]$ and $\beta \in [0,1]$ by definition of $\widetilde{R}$. We consider two situations for the above quadratic equation. First, suppose $\alpha_1$ and $\alpha_2$ are two conjugate complex roots. Then $\alpha_1= \bar{\alpha_2}$, $\alpha_1 + \alpha_2 = \tau\beta+\beta -1$ and $\alpha_1\alpha_2 = \rfrc{1}{2}\tau \beta - \rfrc{1}{2}\beta + \rfrc{1}{4}$ such that
		\begin{equation*}
		|\alpha|^2 = |\alpha_1\bar{\alpha_1}| = |\alpha_1\alpha_2| = \left|\rfrc{1}{2}\tau \beta - \rfrc{1}{2}\beta + \rfrc{1}{4}\right| \leq \frac{1}{4}
		\end{equation*}
		which gives $|\alpha|\leq \rfrc{1}{2}$. The second situation is that two roots are real numbers and must look like
		\begin{equation*}
		\alpha_1 = \frac{1+\tau}{2}\beta + \frac{\sqrt{\beta}\sqrt{(1+\tau)^2\beta-4\tau}}{2} - \rfrc{1}{2} \quad
		\alpha_2 =  \frac{1+\tau}{2}\beta - \frac{\sqrt{\beta}\sqrt{(1+\tau)^2\beta-4\tau}}{2} - \rfrc{1}{2}.
		\end{equation*}
		
		To get the largest possible value of $\alpha$, we only look at $\alpha_1$ because $\alpha_1 \geq \alpha_2$ for any fixed $\beta$. Since $\alpha_1$ is an increasing function of $\beta$ and $\beta \in [0,1]$, the largest real value of $\alpha$ should $\rfrc{1}{2}$ when $\beta = 1$. On the other hand, to get the smallest negative real value of $\alpha$, we only need to look at $\alpha_2$.     One can write $\alpha_2 = \frac{1+\tau}{2}(\beta - \sqrt{\beta^2 - \frac{4\tau}{(1+\tau)^2}})-\rfrc{1}{2}$ to see that $\alpha_2\geq-\rfrc{1}{2}$. So we conclude that if $\alpha$ is real, then $-\rfrc{1}{2} \leq \alpha \leq \rfrc{1}{2}$.
		
		Under both situations, one can conclude $\alpha$ must satisfy $|\alpha|\leq \rfrc{1}{2}$, lying in a disk centered at $0$ with radius $\rfrc{1}{2}$, i.e. $\Dcal(0,\rfrc{1}{2})$. So the eigenvalues of $N$ must lie in the disk $\Dcal(\rfrc{1}{2},\rfrc{1}{2})$ by the shift. Consequently, the only possible eigenvalue on the unit circle is $1$.

		(c). $\gamma_1,\gamma_2$ are the two roots of the quadratic polynomial, i.e. $\gamma^2 - (1+\tau)\gamma\beta + \tau\beta= (\gamma-\gamma_1)(\gamma-\gamma_2)=0$. For given $\beta$, they must satisfy
		$$
		\gamma_1\gamma_2 = \tau\beta  \quad \text{and} \quad \gamma_1 + \gamma_2 = (1+\tau) \beta = \beta + \gamma_1\gamma_2.
		$$
		If $N$ has an eigenvalue $\gamma_1 = 1$, then $\gamma_2 = (1+\tau)\beta - 1 = \beta+\gamma_1\gamma_2 - \gamma_1 = \beta + \gamma_2 -1$, hence $\beta = 1$ must be true and $\widetilde{R}$ has an eigenvalue equal to $1$. Conversely, if $\widetilde{R}$ has an eigenvalue $\beta = 1$, the quadratic polynomial \eqref {FISTA-beta-gamma} will reduce to $ \gamma^2 - (1+\tau)\gamma+ \tau= 0 $, which gives $\gamma_1 = 1$ and $\gamma_2 = \tau$. Then $N$ has an eigenvalue equal to $1$.
		
		(d). Notice in \eqref{FISTA-beta-gamma} that each eigenvalue $\beta$ of $\widetilde{R}$ maps to two eigenvalues of $N$, $\gamma_1$ and $\gamma_2$, and associated eigenvector $\widetilde{\Bw}_2$ of $\widetilde{R}$ maps to two eigenvectors of $N$, $ \bpm \gamma_1 \widetilde{\Bw}_2 \\\widetilde{\Bw}_2 \epm$ and $ \bpm \gamma_2 \widetilde{\Bw}_2 \\\widetilde{\Bw}_2 \epm$. As shown in (c), $N$ has an eigenvalue equal to $1$ if and only if $\widetilde{R}$ has an eigenvalue equal to $1$. Since $R$ and $\widetilde{R}$ have the similar eigenstructure, eigenvalue $1$ of $\widetilde{R}$ must have a complete set of eigenvectors. So the only possible situation that eigenvalue $1$ of $N$ does not have a complete set of eigenvectors is that both $\gamma_1$ and $\gamma_2$ equal to $1$. However, this is impossible because we have shown in (c) that $\beta = 1$ gives $\gamma_1 = 1$ and $\gamma_2 = \tau$ which is close to $1$ but not equal. As a result, if $N$ has an eigenvalue $1$, and then its algebraic and geometric multiplicities coincide.
	\qquad \end{proof}}

\section{Regimes} \label{sec:regimes}
Since the ISTA and FISTA updating steps have been converted into a variation of an eigenproblem in previous sections, we can study the convergence in terms of the spectral properties of the operator $\Ra$ in \eqref{ISTA-Raug} and $\Na$ in \eqref{FISTA-Naug}. Hence in this section, we show how the spectral properties of  $\Ra$, $\Na$ are reflected in the possible convergence ``regimes" that ISTA and  FISTA can encounter.

\subsection{Spectral Properties} \label{sec:spectral}
The eigenvalues of the augmented matrices $\Ra$ and $\Na$ consist of those of $R$, $N$ plus an extra eigenvalue equal to 1, respectively. If $R$ (or $N$) already has an eigenvalue equal to $1$, then the extra eigenvalue $1$ may or may not add a corresponding eigenvector. The next lemma gives limits on the properties of the eigenalue $1$ for any augmented matrix of the general form of  $\Ra$, $\Na$. 

\begin{lemma}\label{lemma:unique}
	Let $\Ma = \bpm M & \Bp \\ 0 & 1 \epm$ be any block upper triangular matrix with a $1\times 1$ lower right block. The matrix $\Ma$ has an eigenvalue $\alpha_1 = 1$ and suppose its corresponding eigenvector has a non-zero last element. We scale that eigenvector to take the form $\bpm \Bw \\ 1 \epm = \Ma \bpm \Bw \\ 1 \epm$.
	If the upper left block $M$ either has no eigenvalue equal to $1$ or the eigenvalue $1$ of $M$ has a complete set of eigenvectors, then $\alpha_1 = 1$ has no non-trivial Jordan block. Furthermore, if the given eigenvector $\bpm \Bw \\ 1 \epm$ is unique, then $M$ has no eigenvalue equal to 1.
\end{lemma}

We refer readers \cite{B13} to the proof of Lemma \ref{lemma:unique}. Now we summarize spectral properties of our specific operators $\Ra$ and $\Na$ in terms of their possible Jordan canonical forms as given in the following lemmas. 
Essentially these lemmas say that all their eigenvalues must have absolute value strictly less than 1, except for the eigenvalue equal to 1. And the eigenvalue 1's geometric multiplicity either equal to or one less than its algebraic multiplicity.

\begin{lemma}\label{lemma:Jordan-decomp-R}
Assuming $\widehat{D}^{[k+1]} = \widehat{D}^{[k]}$, then $\Rak$ in \eqref{ISTA-Raug} is fixed and has a spectral decomposition $\Rak = \BP_{R}\BJ_{R}^{[k]}\BP_{R}^{-1}$, where $\BJ_{R}^{[k]}$ is a block diagonal matrix
\be\label{eq:J-ISTA}
\BJ_{R}^{[k]} = \Diag\left\{ \bpm 1&1\\0&1\epm , I_R^{[k]} , \widehat{\BJ}_R^{[k]}  \right\}
\ee
where any of these blocks might be missing.
Here
$I_R^{[k]} $ is an identity matrix and $\widehat{\BJ}^{[k]}_R$ is a matrix with spectral radius strictly less than $1$.
\end{lemma}

\begin{proof}
For $\Rak$, the upper left block of \eqref{ISTA-Raug} (i.e. $R^{[k]}$) satisfies Lemma \ref{lemma:ISTA-norm} and hence contributes blocks of the form $I_R^{[k]} $, $\widehat{\BJ}_R^{[k]} $. No eigenvalue with absolute value $1$ can have a nondiagonal Jordan block, so the blocks corresponding to those eigenvalues must be diagonal. Embedding that upper left block $R^{[k]} $ into the entire matrix yields a matrix $\Rak$, with the exact same set of eigenvalues with the same algebraic and geometric multiplicities, except for eigenvalue $1$. 

If the upper left block of $\Rak$  has no eigenvalue equal to 1, then $\Rak$ has a simple eigenvalue 1.  In general for eigenvalue 1, the algebraic multiplicity goes up by one and the geometric multiplicity can either stay the same or increase by 1. In other words, $\Rak$ either satisfies the conditions of Lemma \ref{lemma:unique}, or the algebraic and geometric multiplicities of eigenvalue 1 for $\Rak$ differ by 1, meaning we have a single $2\times2$ Jordan block $\bpm 1 & 1 \\ 0 &1 \epm$.
\qquad \end{proof}

\begin{lemma}\label{lemma:Jordan-decomp-N}
	Assuming $\widetilde{D}^{[k+1]} = \widetilde{D}^{[k]}$, since $\Nak$  in \eqref{FISTA-Naug} is different at each step, for each step $k$, there exists a $\BP_{N}^{[k]}$ such that $\Nak$ has a spectral decomposition $\Nak = \BP_{N}^{[k]} \BJ_{N}^{[k]}(\BP_{N}^{[k]})^{-1}$, where $\BJ_{N}^{[k]}$ is the block diagonal matrix:
	\be\label{eq:J-FISTA}
	\BJ_{N}^{[k]} = \Diag\left\{ \bpm 1&1\\0&1\epm , I_N^{[k]}  ,  \widetilde{\BJ}_{N}^{[k]}  \right\}
	\ee
	where any of these blocks might be missing.
	Here $I_N^{[k]} $ is an identity matrix, $\widetilde{\BJ}_{N}^{[k]} $ is a matrix with spectral radius strictly less than $1$.
\end{lemma}
\begin{proof}
	The proof is similar to $\Rak$.  We only note here that the upper left block of \eqref{FISTA-Naug} (i.e. $N$) satisfies Lemma \ref{lemma:FISTA-norm} and hence contributes blocks of the form $I_N^{[k]} $, $\widetilde{\BJ}_{N}^{[k]}$.
\qquad \end{proof}

\subsection{Four regimes} \label{sec:four-regimes}
Lemma \ref{lemma:Jordan-decomp-R} \& \ref{lemma:Jordan-decomp-N} give rise to the four possible ``regimes" associated with the ISTA and FISTA iterations, depending on the flag matrix and the eigenvalues of $\Rak$, $\Nak$. We treat separately the case where the flag matrix remains the same at each iteration, in which there are three possible regimes, and treat all the transitional cases together in their own fourth regime. For simplicity of the notation, we let $D$ denote the flag matrix instead of $\widehat{D}$ and $\widetilde{D}$ unless specified.

When the flag matrix does not change, i.e. $D^{[k+1]} = D^{[k]}$, the ISTA operator $\Rak$ remains invariant over those passes while the FISTA operator $\Nak$ is slightly different at each iteration due to the changing parameter $\tau^{[k]} $. In both cases, the structure of the spectrum for that specific operator controls the convergence behavior of the process during these passes. We summarize as follows three specific possible regimes distinguished by the eigenstructure of the operators $\Rak$, $\Nak$. One of the these three regimes must occur when the
flag matrix is unchanged from one step to the next: $D^{[k+1]} = D^{[k]}$. 

{
[A]. The spectral radius of $R^{[k]}$ (or $N^{[k]}$) is strictly less than 1. The block $\bpm 1 & 1 \\ 0 &1 \epm$ is absent from \eqref{eq:J-ISTA} (or \eqref{eq:J-FISTA}), and the block $I_R^{[k]}$ (or $I_N^{[k]}$) is $1\times1$. If close enough to the optimal solution (if it exists), the result is linear convergence to that solution.

For $\Rak$, as long as the flags remain the same, the recurrence \eqref{ISTA-Raug} hence will converge linearly to a unique fixed point at a rate determined by the next largest eigenvalue in absolute value (largest eigenvalue of the block $\widehat{\BJ}_R^{[k]}$), according to the theory for the power method. If the flags $\widehat{D}^{[k]}$ are consistent with the eigenvector satisfying \eqref{ISTA-def-D}, then the eigenvector must satisfy the KKT condition because of Lemma \ref{ISTA-KT-2}. 

For $\Nak$, though it changes slightly at each step, yet we will show later that the left and right eigenvectors
for eigenvalue 1 do not depend on $\tau$ (Lemma \ref{lemma:FISTA-simple}), and the remaining eigenvalues remain
bounded away from 1.  The result is that we observe linear convergence
to a unique fixed point with the slightly changing convergence rate.  If the flags $\widetilde{D}^{[k]}$ are consistent with the eigenvector satisfying  \eqref{FISTA-def-D}, then that eigenvector must satisfy the KKT condition because of Lemma \ref{FISTA-KT-2}. 

[B]. $R^{[k]}$ (or $N^{[k]}$) has an eigenvalue equal to 1 which results in a $2\times2$ Jordan block $\bpm 1 & 1 \\ 0 &1 \epm$ for $\Rak$ (or $\Nak$). Therefore, the iteration process tends to a constant step. 

For $\Rak$, the theory of power method implies that the vector iterates will converge to the invariant subspace corresponding to the largest eigenvalue $1$. The presence of $\bpm 1 & 1 \\ 0 &1 \epm$ means that there is a Jordan chain: two non-zero vectors $\Bq, \Br$ such that $(\Ra - I)\Bq = \Br$, $(\Ra - I)\Br = 0$. Any vector which includes a component of the form $\alpha \Bq + \beta \Br$ will be transformed into $\Ra(\alpha \Bq + \beta \Br) = \alpha \Bq + (\alpha + \beta) \Br$, i.e. each pass would add a constant vector $\alpha \Br$, plus fading lower order terms from the other lesser eigenvalues. As long as the flags do not change, this will result in constant steps: the difference between consecutive iterates, $\bpm \widehat{\Bw}^{[k+1]} \\ 1 \epm- \bpm \widehat{\Bw}^{[k]} \\ 1 \epm$, would converge to a constant vector, asymptotically as the effects of the smaller eigenvalues fade. That constant vector is an eigenvector for eigenvalue $1$. The ISTA iteration will not converge until a flag change in $\widehat{\Bw}^{[k]}$ forces a change in the flags $\widehat{D}^{[k]}$. If we satisfy the conditions for global convergence of ISTA, then such a flag change is guaranteed to occur. 

The same situation applies to $\Na$. The difference between two iterates $\bpm \widetilde{\Bw}^{[k+1]} \\ 1 \epm- \bpm \widetilde{\Bw}^{[k]} \\ 1 \epm$, would asyptotically converge to a constant vector.  The FISTA iteration will not converge until a flag change in $\widetilde{\Bw}^{[k]}$ forces a change in the flags $\widetilde{D}^{[k]}$. Such a flag change is guaranteed to occur due to the global convergence of FISTA. In Section \ref{sec:regime-b}, we will show that FISTA can jump out of such regime very fast, which is the main reason why it is faster than ISTA. See Section \ref{sec:examples} for more discussions on its numerical behavior.

[C]. $R^{[k]}$ (or $N^{[k]}$) has an eigenvalue equal to $1$, but the block $\bpm 1 & 1 \\ 0 &1 \epm$ is absent. For $\Ra$ (or $\Na$), the convergence rate of this regime will still depend on $\rho(\widehat{\BJ}_R^{[k]})$ and $\rho(\widetilde{\BJ}_N^{[k]})$. If we assume the solution is unique, the eigenvalue $1$ of $\Ra$ (or $\Na$) must be simple by Lemma \ref{lemma:unique}. So the iteration will eventually jump out this type of regime. 

When the flag matrix does change, it means the set of active constraints at the current passes in the process has changed, and the current pass is a transition to a different operator with a different eigenstruture. 

[D]. The operator $R^{[k+1]}$ (or $N^{[k+1]}$) will be different from $R^{[k]}$ (or $N^{[k]}$) due to different flag matrix. 

}

\section{Unique Solution: Linear Convergence} \label{sec:linear}

{\subsection{Partial Spectral Decomposition} \label{sec:partial}
	As we remarked at the end of Section \ref{sec:FISTArecur}, $\Nak$ is different for different step $k$. In the Lemma \ref{lemma:Jordan-decomp-N}, we show that, for different $k$,  $\Nak$ is spectrally decomposed by different matrix $\BP_{N}^{[k]}$. This section shows that if unique solution is assumed, then for different $k$, $\Nak$ can be spectrally decomposed by the same matrix, denoted as $\BP_{N}$. 
	
	\begin{lemma} \label{lemma:fixed-point1}
		Assume that FISTA in Alg.\ 2 has a unique solution $\Bx^*$. If iteration $j >K$ for some $K$, the stepsize $t^{[j]}$ becomes frozen
		at a constant value $t^{[K]}$, i.e. $\tau^{[j]}= \frac{t^{[j-1]}-1}{t^{[j]}} = c \in [0,1]$,  then the iteration will converge to the same solution $\Bx^*$.
	\end{lemma}
	\begin{proof}
		Since $\Bx^*$ is the unique solution, it must be the fixed point of the algorithm. If the algorithm converges, one must have $\Bx^{[k]} - \Bx^{[k-1]} \rightarrow 0$. In Alg.\ 2, it is easy to see that no matter what $c = \frac{t^{[j-1]}-1}{t^{[j]}}$ is, the converging point will always be the same point, the optimal solution $\Bx^*$.
	\qquad \end{proof}
	Since Alg.\ 4 is equivalent to Alg.\ 2, we have the similar statement for $\Nak$, as shown in the next lemma. 
	
	\begin{lemma}  \label{lemma:fixed-point2}
		Assume that the FISTA in Alg.\ 4 has a unique solution $\fwas = \bpm \widetilde{\Bw}^* \\ 1 \epm$. If iteration $j >K$ for some $K$, the stepsize becomes frozen at a constant number, i.e. $\tau^{[j]} = \frac{t^{[j-1]}-1}{t^{[j]}} = c \in [0,1]$,  ${\bf N}^{[j]}_{\bf aug}$ will then become a constant matrix and the iteration will converge to the same solution $\fwas$. In other words, $(1, \fwas)$ is a simple dominant eigenpair of ${\bf N}^{[j]}_{\bf aug}$.
	\end{lemma}
	\begin{proof}
		Once $\tau^{[j]}$ is fixed for $j>K$, then ${\bf N}^{[j]}_{\bf aug}$ will be fixed. Due to Lemma \ref{lemma:fixed-point1},
		Alg.\ 4 will still converge to the same solution $\fwas$. 
	\qquad \end{proof}
	
	The above two lemmas directly indicates the next proposition:
	\begin{proposition}
		Though $\Nak$ in Alg.\ 4 are different for different $k$, they share the same simple dominant eigenpair $(\lambda=1, \fwas)$ as long as the LASSO problem has a unique solution.
	\end{proposition}
	
	Now we show that if every $\Nak$ has the same simple dominant eigenpair, then for different $k$, $\Nak$ can be decomposed by the same matrix, denoted as $\BP_{N}$.
	\begin{lemma}\label{lemma:FISTA-simple}
		Assuming $\widetilde{D}^{[k+1]} = \widetilde{D}^{[k]}$ and  FISTA in Alg.\ 4 has a unique solution,  then $\Nak$ has a spectral decomposition $\Nak = \BP_{N}\BJ_{N}^{[k]} \BP_{N}^{-1} $. Note that $\BP_{N}$ is the same for all $\Nak$ and
		\be \label{eq:J-FISTA-simple}
		\BJ_{N}^{[k]}
		= \bpm 1 & 0 \\ 0& \widetilde{\BJ}_N^{[k]} \epm
		\ee
		where $\widetilde{\BJ}_N^{[k]}$ is a matrix with spectral radius strictly less than $1$. 
	\end{lemma}
	
	\begin{proof}
		If the LASSO problem has a unique solution, by Lemma \ref{lemma:fixed-point1} \& \ref{lemma:fixed-point2}, for any $k$, $\Nak$ must share the same single dominant eigenpair $(\lambda=1, \fwas)$.
		This is a simple eigenvalue with a fixed left eigenvector $\Bz^{\T} = (0,\cdots,0,1)$.
		Therefore, we can construct   $\BP_N = \left[  \fwas, \BW\right]$ where $\BW \in \R^{(2n+1) \times 2n} $ forms the basis for space $\{\Bz\}^{\perp}$ and its inverse $\BP_N^{-1}$ has the form $\BP_N^{-1} = \left[ \Bz, \BZ\right]^{\T}$ by scaling, where $\BZ$ is a basis for $(\fwas)^{\perp}$ determined by $ \fwas$, $\Bz$ and $\BW$.
		So $\BP_N$ is a matrix independent of any $\tau^{[k]}$ that for $\Nak$, 
		\begin{equation*}
		\BP_{N}\Nak \BP_{N}^{-1} = \left[ \fwas, \BW\right] \Nak\left[ \Bz, \BZ\right]^{\T} = \bpm 1 & 0 \\ 0& \widetilde{\BJ}_N^{[k]} \epm.
		\end{equation*}
        \qquad \end{proof}

\begin{lemma}\label{lemma:FISTA-tau-1}
	We denote matrix ${\bf N}_{\bf aug}'$ as $\Nak$  in which $\tau = 1$, then one can write $\Nak = {\bf N}_{\bf aug}' + (1-\tau^{[k]}) {\bf \barN}_{\bf aug}$, where 
	\be
	\begin{aligned}
	& &{\bf N}_{\bf aug}' = \bpm 2\widetilde{R}^{[k]} & -\widetilde{R}^{[k-1]} & \rfrc{\lambda}{L}[-2\widetilde{\Bd}^{[k]} +\widetilde{\Bd}^{[k-1]}] \\ I & 0 & 0 \\ 0 & 0 & 1 \epm \\ 
	\text{and} & \quad &{\bf \barN}_{\bf aug} =\bpm -\widetilde{R}^{[k]} & \widetilde{R}^{[k-1]} & \rfrc{\lambda}{L}[\widetilde{\Bd}^{[k]} - \widetilde{\Bd}^{[k-1]}]\\ 0 & 0 & 0 \\ 0 & 0 & 0 \epm.
	\end{aligned}
	\ee 
	
	If $D^{[k]} = D^{[k-1]}$ and $\Nak$ has only a simple eigenvalue equal to $1$, then ${\bf N}_{\bf aug}'$ must also have a simple eigenvalue equal to $1$ and $\Nak = \BP_{N}\BJ_{N}^{[k]} \BP_{N}^{-1} =\BP_{N}\BJ_{N'}  \BP_{N}^{-1} + (1-\tau^{[k]}) \BP_{N}\BJ_{\barN} \BP_{N}^{-1} $ with the same $\BP_{N}$ in Lemma \ref{lemma:FISTA-simple}. Note that
	$
	\BJ_{N'} 
	= \bpm 1 & 0 \\ 0& \widetilde{\BJ}_{N'} \epm
	$
	and 
	$
	\BJ_{\barN}
	= \bpm 0 & 0 \\ 0& \widetilde{\BJ}_{\barN} \epm
	$, 
	where $\widetilde{\BJ}_{N'}$ and  $\widetilde{\BJ}_{\barN}$ are matrices with spectral radius strictly less than $1$. Consequently,  $\widetilde{\BJ}_N^{[k]} = \widetilde{\BJ}_{N'}+(1-\tau^{[k]}) \widetilde{\BJ}_{\barN}$. 
\end{lemma}

\begin{proof}
	By Lemma \ref{lemma:FISTA-norm}(b), eigenvalue of ${\bf N}_{\bf aug}'$ must lie in the disk $\Dcal(\rfrc{1}{2}, \rfrc{1}{2})$. Besides, now that $\Nak$ has only a simple eigenvalue equal to $1$, by Lemma \ref{lemma:FISTA-norm}(c), $\widetilde{R}^{[k]}$ should have no eigenvalue equal to $1$. This indicates matrix	$\bpm 2\widetilde{R}^{[k]} & -\widetilde{R}^{[k-1]} \\ I & 0\epm$ has no eigenvalue equal to $1$. And hence ${\bf N}_{\bf aug}'$ must also have a simple eigenvalue equal to $1$. 
	
	Due to Lemma \ref{lemma:fixed-point2}, $(1, \fwas)$ is in the right nullspace of ${\bf N}_{\bf aug}' - I$. Additionally, $(0,\cdots,0,1)^{\T}$  is in the left nullspace of ${\bf N}_{\bf aug}' - I$. Therefore, one can construct the same $\BP_{N}$ in Lemma \ref{lemma:FISTA-simple} such that 
	${\bf N}_{\bf aug}'  = \BP_{N} \bpm 1 & 0 \\ 0& \widetilde{\BJ}_{N'} \epm\BP_{N}^{-1}$. According to the relation $\Nak = {\bf N}_{\bf aug}' + (1-\tau^{[k]}) {\bf \barN}_{\bf aug}$, this $\BP_{N}$ must work as well for eigenvalue $0$ of ${\bf \barN}_{\bf aug}$.
\qquad \end{proof} }
\subsection{Local Linear Convergence} \label{sec:local}
In the case that \eqref{LASSO} has a unique solution with strict complementarity, we can give a guarantee that eventually the flag matrix will not change. By strict complementarity, we mean that for every index $i$, $\widehat{w}^*_i \neq \pm \rfrc{\lambda}{L}$(or $\widetilde{w}^*_i \neq \pm \rfrc{\lambda}{L}$). Once the flag matrix stays fixed, the ISTA (or FISTA) iteration behaves just like the power method (or similar to power method) for the matrix eigenvalue problem. In this case, the spectral theory developed here gives a guarantee of linear convergence with the rate equal to the second largest eigenvalue of the matrix operator.

In this section we will use the $l_\infty$ norm of a vector: $\|\Bv\|_{\infty}=\max_i |v_i|$, and the associated induced matrix norm $\|A\|_{\infty}= \max_i\sum_j|a_{ij}|$. We will also use the matrix G-norm where G is a non-singular matrix, defined to be $\|\Bx\|_G = \|G\Bx\|_{\infty}$ for any vector $\Bx$, and $\|A\|_G = \|GAG^{-1}\|_{\infty}$ for any matrix A. The following lemma relates the vector $\infty$-norm to the vector $2$-norm. 

\begin{lemma}
For any n-vectors $\Ba, \Bb$, $(\|\Ba\|_{\infty}+\|\Bb\|_{\infty})^2 \leq 2(\|\Ba\|_2^2+\|\Bb\|_2^2)$
\end{lemma}

We refer readers to \cite{B13} for the proof. 

Under the assumption of strict complementarity and unique solution, we can prove the specific result that ISTA (or FISTA) iteration must eventually reach and remain in ``linear convergence" regime [A]. First we note that by Lemma \ref{ISTA-KT-2} (or \ref{FISTA-KT-2}), this solution must correspond to a unique eigenvector of eigenvalue $ 1$ for the matrix $\Ra$ (or $\Na$) when the flag matrix $D^{[k]} = D^{[k+1]}$ does not change. Additionally, by Lemma \ref{lemma:unique}, the matrix $R$ (or $N$) has no eigenvalue equal to 1, and by Lemma \ref{lemma:ISTA-norm} (or \ref{lemma:FISTA-norm}) all the eigenvalues must be strictly less than 1 in the absolute value. Hence the following lemma applies to this situation.

\begin{lemma}\label{lemma:6.2}
Consider the general augmented matrix and its eigenvector
\be
\Ma = \bpm M & \Bp \\ 0 & 1 \epm \text{and} \quad  {\bf w}^{*}_{\bf aug} = \bpm \Bw^{*} \\ 1 \epm \text{such that}\quad \Ma  {\bf w}^{*}_{\bf aug} =  {\bf w}^{*}_{\bf aug}
\ee
where $M$ is any $n\times n$ matrix such that the spectral radius $\rho$ of $M$ satisfies $\rho(M) < 1$. The vector $\Bw^*$ is the unique eigenvector corresponding to eigenvalue 1, scaled so that its last element is $\Bw_{n+1}^*=1$. Then the following holds.

{\em (a)}. For any $\epsilon > 0$, there is a matrix norm $\|\cdot\|_G$ such that $\rho(M) \leq \|M\|_G < \rho(M) + \epsilon$. In particular, one can choose $G$ so that $\|M\|_G <1 $. Also, there is a positive constant $C_1$ (depending on $M$ \& $G$) such that for any vector or matrix $X$, $\|X\|_G \leq C_1 \|X\|_{\infty}$ and $\|X\|_{\infty} \leq C_1 \|X\|_G$.

{\em (b)}. The iterates of the power iteration ${\bf w}^{[k+1]}_{\bf aug} = \Ma \wak$ satisfy $\|\wak - {\bf w}^{*}_{\bf aug}\|_G \leq \|M\|_G^{k} \|{\bf w}^{[0]}_{\bf aug} - {\bf w}^{*}_{\bf aug}\|_G$ and hence converge linearly to ${\bf w}^{*}_{\bf aug}$ at a rate bounded by $\rho(M) + \epsilon$ where $\epsilon$ is the same arbitrary constant used in {\em (a)}. This is a special case of the theory behind the power method for computing matrix eigenvalues.

{\em (c)}. Given any positive constant $C_2$, if ${\bf w}^{[0]}_{\bf aug}$ is any vector such that $\|{\bf w}^{[0]}_{\bf aug}-{\bf w}^{*}_{\bf aug}\|_{\infty} \leq C_2/ C_1^2$ then $\|(\Ma)^k {\bf w}^{[0]}_{\bf aug}- {\bf w}^{*}_{\bf aug}\|_{\infty} \leq C_2$ for all $k$. In particular, if $w^*_i$ is bounded away from two points $ \pm \rfrc{\lambda}{L}$ and $C_2= \min_i \{|w_i^* - \rfrc{\lambda}{L}| - \epsilon, |w_i^*+ \rfrc{\lambda}{L}| - \epsilon\}>0$, then any element of vector $(\Ma)^k {\bf w}^{[0]}_{\bf aug}$ should be bounded away from two points $ \pm \rfrc{\lambda}{L}$ for all $k=0,1,2,\cdots$.
\end{lemma}

\begin{proof}
The proof of (a) and (b) are omitted. They are similar to those of the lemma 6.2 in \cite{B13}. For (c), we make more comments. Define $G_{\bf aug} = \bpm G & 0 \\ 0&1 \epm$ with the $G$ from part (a), and define the corresponding $G_{\text{aug}}$-norm on the augmented quantities. Define the following balls around the eigenvector ${\bf w}^{*}_{\bf aug}$:
\be
\begin{aligned}
&\Bcal_1 & = & \{\wa: \|\wa - {\bf w}^{*}_{\bf aug}\|_{\infty} \leq C_2 \} \\
&\Bcal_2 & = & \{\wa: \|\wa - {\bf w}^{*}_{\bf aug}\|_{G_{\bf{aug}}} \leq C_2/C_1 \} \\
&\Bcal_3 & = & \{\wa: \|\wa - {\bf w}^{*}_{\bf aug}\|_{\infty} \leq C_2/C_1^2 \}.
\end{aligned}
\ee
From part (a), $\Bcal_3\subseteq \Bcal_2\subseteq \Bcal_1$. From part (b), if any power method iterate satisfies $\Bw^{[0]}\in \Bcal_2$, then all subsequent iterates stay in $\Bcal_2$. Hence if the power method starts in $\Bcal_3$, all subsequent iterates will lie in $\Bcal_1$.
\qquad \end{proof}

We now invoke the global convergence property of ISTA and FISTA.

\begin{theorem}\label{thm:conv}
If problem \ref{LASSO} is solvable, i.e. $X^*:= \argmin F(\Bx)$, where $F(\Bx) = \rfrc{1}{2}\|A\Bx - \Bb\|_2^2 + \lambda \|\Bx\|_1$. If let $\widehat{\Bx}^{[0]}$ be any starting point in $\R^n$ and $\widehat{\Bx}^{[k]}$ be the sequence generated by ISTA. Then for any $k \geq 1$, $\forall \widehat{\Bx}^* \in X^*$, $F(\widehat{\Bx}^{[k]}) - F(\widehat{\Bx}^{*})$ decreases at the rate of $O(\rfrc{1}{k})$. On the other hand, if let $\By^{[1]}=\widetilde{\Bx}^{[0]}$ be any starting point in $\R^n$, $t^{[1]} = 1$ and $\{\widetilde{\Bx}^{[k]}\}, \{\By^{[k]}\}, \{t^{[k]}\}$ be the sequence generated by FISTA. Then for any $k \geq 1$, $\forall \widetilde{\Bx}^* \in X^*$, we have $F(\widetilde{\Bx}^{[k]}) - F(\widetilde{\Bx}^*)$ decreases at the rate of $O(\rfrc{1}{k^2})$. 
\end{theorem}

This is a restatement of the convergence theorem in \cite{BT09}. It says little on the local behavior of the algorithm. Under the assumption of the unique solution, it guarantees that eventually the iterates will converge to the optimal value. With the iterates convergence, we present our main results in the following theorems.

\begin{theorem}\label{thm:linear-conv-ista}
Suppose the LASSO problem \ref{LASSO} has a unique solution and this solution has strict complementarity: that is for every index $i$, $\widehat{w}^*_i \neq \pm \rfrc{\lambda}{L}$. Then eventually the ISTA iteration reaches a stage where it converges linearly to that unique solution.
\end{theorem}

\begin{proof}
Since for any index $i$, $\widehat{w}^*_i  \neq \pm \rfrc{\lambda}{L}$, $\widehat{w}_i^{[j]}$ (where $j$ is the pass number) could only be in one of three cases: $\widehat{w}_i^{[j]}>\rfrc{\lambda}{L}$, $\widehat{w}_i^{[j]}<-\rfrc{\lambda}{L}$ or $-\rfrc{\lambda}{L}\leq \widehat{w}_i^{[j]}\leq \rfrc{\lambda}{L}$. We can let $C_2= \min \{|\widehat{w}_i^* - \rfrc{\lambda}{L}| - \epsilon, |\widehat{w}_i^*+ \rfrc{\lambda}{L}| - \epsilon\}>0$ for a positive constant $\epsilon$ sufficiently small to make $C_2>0$.
	
By theorem \ref{thm:conv}, there exists a pass $k$ such that $\|\widehat{\Bw}^{[k]} - \widehat{\Bw}^*\|_{\infty}^2  < (C_2/C_1^2)$. Hence $\widehat{\Bw}^{[k]}$ lies in $\Bcal_3$ stated in Lemma \ref{lemma:6.2}(c), and $D^{[k]} = \textsc{{Diag}}(\textsc{sign} (\text{Shr}_{\rfrc{\lambda}{L}}(\widehat{\Bw}^{[k]})))$ is the associated flag matrix. By Lemma \ref{lemma:6.2}(c), there exists a pass $K$ such that $\widehat{\Bw}^{[k]}$ lies in $\Bcal_1$ (i.e. $\| \widehat{\bf w}^{[k]}_{\bf aug} - \widehat{\bf w}^{*}_{\bf aug}\|_{\infty} \leq C_2$) for all $k>K$. This means that $\widehat{w}_i^k$ will remain in one of three cases: $\widehat{w}_i^k>\rfrc{\lambda}{L}$, $\widehat{w}_i^k<-\rfrc{\lambda}{L}$ or $-\rfrc{\lambda}{L}\leq\widehat{w}_i^k\leq\rfrc{\lambda}{L}$ and will never change to another case for all $k>K$. This, combined with the definition of flag matrix $\widehat{D}^{[k]} = \textsc{{Diag}}(\textsc{sign} (\text{Shr}_{\rfrc{\lambda}{L}}(\widehat{\Bw}^{[k]})))$, implies that the flag matrix remain unchanged for all $k>K$. Thus starting at the $K$-th pass, the ISTA iteration reduces to the power method on the matrix $\Rak$, converging linearly to the unique eigenvector at a rate given by Lemma \ref{lemma:6.2}(b). 
\qquad \end{proof}

\begin{theorem}\label{thm:linear-conv-fista}
	Suppose the LASSO problem \eqref{LASSO} has a unique solution and this solution has strict complementarity: that is for every index $i$, $\widetilde{w}^*_i \neq \pm \rfrc{\lambda}{L}$. Then eventually the FISTA iteration reaches a stage where it converges linearly to that unique solution.
\end{theorem}

\begin{proof}
	As the proof for Theorem \ref{thm:linear-conv-ista}, one can construct $C_2'= \min \{|\widetilde{w}_i^* - \rfrc{\lambda}{L}| - \epsilon', |\widetilde{w}_i^*+ \rfrc{\lambda}{L}| - \epsilon'\}$ for a positive constant $\epsilon'$ sufficiently small to make $C_2' > 0$. By Theorem \ref{thm:conv} and Lemma \ref{lemma:6.2}(c), there exists a pass $K_1$ such that $\widetilde{\Bw}^{[k]}$ lies in $\Bcal_1'$ (i.e. $\| \widetilde{\bf w}^{[k]}_{\bf aug} - \widetilde{\bf w}^{*}_{\bf aug}\|_{\infty} \leq C_2$) for all $k>K_1$. It also means the flag matrix will remain unchanged for all $k>K_1$. Thus, by Lemma \ref{lemma:FISTA-simple}
	\be \label{eq: final-w}
	\begin{aligned}
    \widetilde{\bf w}^{[k+l]}_{\bf aug}
  & ={\bf N}^{[k+l-1]}_{\bf aug}  {\bf N}^{[k+l-2]}_{\bf aug}  \cdots  {\bf N}^{[j]}_{\bf aug} \widetilde{\bf w}^{[k]}_{\bf aug} \\
    & =\BP_{N} {\BJ}_N^{[k+l-1]}  \BP_{N} ^{-1} \BP_{N} \widetilde{\BJ}_N^{[k+l-2]}\BP_{N} ^{-1}  \cdots \BP_{N}  \widetilde{\BJ}_N^{[k]}  \BP_{N} ^{-1} \widetilde{\bf w}^{[k]}_{\bf aug} \\
    & = \BP_{N} \bpm 1 & 0 \\ 0& \widetilde{\BJ}_N^{[k+l-1]} \epm \bpm 1 & 0 \\ 0& \widetilde{\BJ}_N^{[k+l-2]} \epm  \cdots\bpm 1 & 0 \\ 0& \widetilde{\BJ}_N^{[k]} \epm \BP_{N} ^{-1} \widetilde{\bf w}^{[k]}_{\bf aug}\\ 
    & =  \BP_{N} \bpm 1 & 0 \\ 0& \widetilde{\BJ}_N^{[k+l-1]} \widetilde{\BJ}_N^{[k+l-2]}   \cdots\widetilde{\BJ}_N^{[k]} \epm \BP_{N} ^{-1} \widetilde{\bf w}^{[k]}_{\bf aug} .
	\end{aligned}
	\ee
	For each $\widetilde{\BJ}_N^{[k]}$, we can write $\widetilde{\BJ}_N^{[k]} = \widetilde{\BJ}_{N'} + (1-\tau^{[k]}) \widetilde{\BJ}_{\barN}$, where $\widetilde{\BJ}_{N'}$, $\widetilde{\BJ}_{\barN}$ are defined in Lemma \ref{lemma:FISTA-tau-1}.
	Due to the fixed flag matrix after $K_1$-th pass, $\widetilde{\BJ}_{N'}$ and $\widetilde{\BJ}_{\barN}$ remain fixed for all $k>K_1$.
	By Lemma \ref{lemma:6.2}, $\forall \epsilon >0$ with $\epsilon < 1 - \rho(\widetilde{\BJ}_{N'})$, $\exists ~  \|\cdot\|_{G}$ such that $ \|\widetilde{\BJ}_{N'}\|_{G} <
	\rho(\widetilde{\BJ}_{N'})+\rfrc12 \epsilon < 1-\rfrc12\epsilon$.
	Since $\widetilde{\BJ}_{\barN}$ is fixed, there must exist a pass $K_2 ( >K_1)$ such that $(1-\tau^{[k]})
	\cdot \|\widetilde{\BJ}_{\barN}\|_G < \rfrc12\epsilon$  for all $k>K_2$.
	Therefore, starting at $K_2$-th pass, one has $\|\widetilde{\BJ}_N^{[k]}\|_G \leq \|\widetilde{\BJ}_{N'}\|_G + (1-\tau^{[k]}) \|\widetilde{\BJ}_{\barN}\|_G <
	\rho(\widetilde{\BJ}_{N'})+\epsilon <  1$.
	\be
	\begin{aligned}
	&\BP_{N} ^{-1} \widetilde{\bf w}^{[k+1]}_{\bf aug} &=& \bpm 1 & 0 \\ 0& \widetilde{\BJ}_N^{[k]} \epm(\BP_{N} ^{-1} \widetilde{\bf w}^{*}_{\bf aug} + \By^{[k]}_{\bf aug}) \\
	&&= &\BP_{N} ^{-1} \widetilde{\bf w}^{*}_{\bf aug} +\bpm 1 & 0 \\ 0& \widetilde{\BJ}_N^{[k]} \epm\By^{[k]}_{\bf aug} = \BP_{N} ^{-1} \widetilde{\bf w}^{*}_{\bf aug}  + \By^{[k+1]}_{\bf aug}
	\end{aligned}
	\ee
	\be
	\|\By^{[k+l]}_{\bf aug} \|_G \leq O(\|\widetilde{\BJ}_{N}^{[k+l]}\|\|\widetilde{\BJ}_{N}^{[k+l-1]}\|\cdots) \leq O((\rho(\widetilde{\BJ}_{N'})+\epsilon)^{[k+l]}) \longrightarrow 0.
	\ee 
	Therefore, $\| \widetilde{\bf w}^{[k]}_{\bf aug}- \widetilde{\bf w}^{*}_{\bf aug}\|$ converge linearly at a rate bounded by $ \rho(\widetilde{\BJ}_{N'})+\epsilon <1$.
\qquad \end{proof}


\section{Acceleration} \label{sec:accel}
It is known that FISTA exhibits a global convergence rate $O(\rfrc{1}{k^2})$, which accelerates ISTA's convergence rate $O(\rfrc{1}{k})$. Compared to this worst case convergence result, we analyze how FISTA and ISTA behave through all iterations on the perspective of spectral analysis we establish in this paper.  First, we characterize one important property based on three possible regimes.
\begin{lemma}
Suppose $R$ and $N$ have the same the flag matrix, ISTA and FISTA have the following relations:

{\em (a)}. If FISTA is in regime [A] or [C], then so is ISTA, and vice versa.

{\em (b)}. If FISTA is in regime [B], then so is ISTA, and vice versa.
\end{lemma}

\begin{proof}
We note that if FISTA and ISTA start at the same iterate, we have
$\widehat{D} = \widetilde{D}$, hence $\widetilde{R}$ defined in \eqref{eq-S-P-Q} is exactly operator $R$ defined in \eqref{ISTA-Raug}. 

(a). If FISTA is in regime [A] or [C], then $N$ either has no eigenvalue equal to $1$ or has a complete set of eigenvectors associated with eigenvalue $1$. In other words, the augmented matrix $\Na$ must have a complete set of eigenvectors for eigenvalue $1$. Let $\bpm \Bx \\\By \\1\epm$ be the eigenvector for eigenvalue $1$, then
\be
\begin{array}{lll}
\multicolumn{3}{l}{
(N-I)\bpm \Bx \\\By \\1\epm = \bpm (1+\tau) R & -\tau R & \widehat{\Bh} \\ I & -I & 0 \\ 0 & 0 & 0 \epm \bpm \Bx \\\By \\1\epm = 0 }\\
&\iff& \Bx = \By \quad \text{(by second row)}\\
&\iff& R\Bx -\Bx +\widehat{\Bh} = (R-I)\Bx +\widehat{\Bh} =0\\
&\iff& \bpm R & \widehat{\Bh} \\0&1 \epm \bpm \Bx \\ 1\epm =\bpm \Bx \\ 1\epm.
\end{array}
\ee
Therefore, $\bpm \Bx \\ 1\epm$ becomes the eigenvector for eigenvalue $1$ of $\Ra$. $R$ must either have no eigenvalue equal to $1$ (in regime [A]) or have a complete set of eigenvectors associated with eigenvalue $1$ (in regime [C]). The opposite direction follows by similar argument.

(b). Since one of the regimes [A], [B], [C] must occur, this statement can be considered as the contraposition of (a). 
\qquad \end{proof}

This lemma suggests that both ISTA and FISTA are in the same regime as long as both operators have the same flag matrix. It motivates one to compare in each regime between FISTA and ISTA when starting from the same starting point (which results in the same flag matrix).
By assuming the same starting point and a fixed flag matrix, we have
$\widehat{D}^{[k]} = \widetilde{D}^{[k]} = \widehat{D}^{[k+1]}= \widetilde{D}^{[k+1]}$ and thus $\widetilde{R} = R$, $\widehat{\Bh} = \bar{\Bh}$. We will use these notations interchangeably and omit $^{[k]}$ for anything but iterates in the following analysis. It turns out that FISTA is faster in regime [B], but not always faster in regimes [A] and [C] depending on the parameter $\tau^{[k]}$.

\subsection{In Regime [B]} \label{sec:regime-b}
In regime [B], as mentioned in Section \ref{sec:local}, there exist Jordan chains such that the difference of iterates will converge to a constant step. Let $\bpm \widehat{\Bw}^{[k+1]} \\ 1 \epm$ , $\bpm \widehat{\Bw}^{[k]} \\ 1 \epm$ and $\bpm \widetilde{\Bw}^{[k+1]} \\ \widetilde{\Bw}^{[k]} \\ 1 \epm$, $\bpm \widetilde{\Bw}^{[k]} \\ \widetilde{\Bw}^{[k-1]} \\ 1 \epm$ be two consecutive iterates for ISTA and FISTA, respectively. In the following lemmas, we will show that the constant step for FISTA is larger than ISTA when starting at the same point, which yield a speedup.

\begin{lemma}\label{lemma:ISTA-constant-step}
The constant step vector for ISTA is $\bpm \Bv \\ 0 \epm$, where $\Bv = R \Bv$ is an eigenvector of $R$.
\end{lemma}

\begin{proof}
For ISTA, there must be a Jordan block $\BJ_R^1$ for the augmented matrix $\Ra$. Then there exists a Jordan chain such that
\be
\bpm R & \widehat{\Bh} \\0&1 \epm = \bpm \widehat{\Bw} \\ 1 \epm = \bpm \widehat{\Bw}+\Bv \\ 1 \epm \quad \text{and} \quad \bpm R & \widehat{\Bh} \\0&1 \epm \bpm \Bv \\ 0 \epm = \bpm \Bv \\ 0 \epm.
\ee
In other words, each pass of ISTA will add a constant vector $\bpm \Bv \\ 0 \epm$ in regime [B].
\qquad \end{proof}

\begin{lemma}\label{lemma:FISTA-constant-step}
The constant step vector for FISTA has the form $\bpm c\Bv \\ c\Bv \\ 0 \epm$, where $\Bv$ is the same $\Bv$ in lemma  \ref{lemma:ISTA-constant-step}, $c$ is a scalar to be determined.
\end{lemma}

\begin{proof}
Assume the constant vector is $\bpm \Bv_1 \\ \Bv_2 \\ 0 \epm$. Then basic iteration of FISTA is
\begin{equation*}
\bpm \widetilde{\Bw}^{[k+1]} \\ \widetilde{\Bw}^{[k]} \\ 1 \epm = N \bpm \widetilde{\Bw}^{[k]} \\ \widetilde{\Bw}^{[k-1]} \\ 1 \epm=\bpm (1+\tau) R & -\tau R & \widehat{\Bh} \\ I & 0 & 0 \\ 0 & 0 & 1 \epm \bpm \widetilde{\Bw}^{[k]} \\ \widetilde{\Bw}^{[k-1]} \\ 1 \epm = \bpm \widetilde{\Bw}^{[k]} \\ \widetilde{\Bw}^{[k-1]} \\ 1 \epm + \bpm \Bv_1 \\ \Bv_2 \\ 0 \epm.
\end{equation*}

Due to the presence of Jordan block $\bpm 1 & 1\\ 0 & 1\epm$, there exists a Jordan chain
\be\label{FISTA:Jordan-chain}
(N-I)\bpm \widetilde{\Bw}^{[k]} \\ \widetilde{\Bw}^{[k-1]} \\ 1 \epm = \bpm \Bv_1 \\ \Bv_2 \\ 0 \epm \quad \text{and} \quad N\bpm \Bv_1 \\ \Bv_2 \\ 0 \epm = \bpm \Bv_1 \\\Bv_2 \\ 0 \epm.
\ee

In the second equation of \eqref{FISTA:Jordan-chain}, the second row implies $\Bv_1 = \Bv_2$. Then, the first row implies $R\Bv_1 = \Bv_1$. Since both $\Bv_1$ and $\Bv$ are eigenvectors for eigenvalue $1$ of $R$, we can write $\Bv_1 = c\Bv$ where $c$ is a scalar to be determined. Hence the constant step should be $\bpm c\Bv \\ c\Bv  \\ 0 \epm$.
\qquad \end{proof}

\begin{lemma}\label{lemma-fista-acce}
Suppose ISTA and FISTA start from the same point in the same regime [B], i.e.  $\widehat{\Bw}^{[k]} = \widetilde{\Bw}^{[k]} $, then $c$ in Lemma \ref{lemma:FISTA-constant-step} equals to $\frac{1}{1-\tau}$, where $\tau$ is a scalar close to $1$. The constant step vector for FISTA is $\frac{1}{1-\tau}\bpm \Bv \\ \Bv\\ 0 \epm$, which is larger than $\bpm \Bv \\ \Bv \\ 0 \epm$, the ISTA constant step.
\end{lemma}
\begin{proof}
By Lemma \ref{lemma:FISTA-constant-step}, the equation \eqref{FISTA:Jordan-chain} expands to
\be
\begin{aligned}
	&(N-I)\bpm \widetilde{\Bw}^{[k]} \\ \widetilde{\Bw}^{[k-1]} \\ 1 \epm &=& \left[\bpm (1+\tau) R & -\tau R & \widehat{\Bh} \\ I & 0 & 0 \\ 0 & 0 & 1 \epm - I \right]\bpm \widetilde{\Bw}^{[k]} \\ \widetilde{\Bw}^{[k-1]} \\ 1 \epm \\
	&& =& \bpm ((1+\tau)R-I)\widetilde{\Bw}^{[k]} - \tau R \widetilde{\Bw}^{[k-1]}+\widehat{\Bh} \\ \widetilde{\Bw}^{[k]} - \widetilde{\Bw}^{[k-1]} \\ 0 \epm
\end{aligned}
\ee
which is supposed to be equal to $\bpm c\Bv \\ c\Bv \\ 0 \epm$. From the second row, $\widetilde{\Bw}^{[k]} -\widetilde{\Bw}^{[k-1]} = c\Bv$ or $\widetilde{\Bw}^{[k-1]}= \widetilde{\Bw}^{[k]} -c\Bv$. Hence, the first row should be $c\Bv = ((1+\tau)R-I)\widetilde{\Bw}^{[k]} - \tau R \widetilde{\Bw}^{[k-1]}+\widehat{\Bh}=((1+\tau)R-I)\widetilde{\Bw}^{[k]} - \tau R (\widetilde{\Bw}^{[k]} - c\Bv)+\widehat{\Bh} = (R-I)\widetilde{\Bw}^{[k]}+\widehat{\Bh}+c\tau\Bv$. The last equality follows by $R\Bv = \Bv$.

If FISTA and ISTA start from the same point $\widehat{\Bw}^{[k]} = \widetilde{\Bw}^{[k]}$, then $c\Bv = (R-I)\widetilde{\Bw}^{[k]}+\widehat{\Bh}+c\tau\Bv = (R-I)\widehat{\Bw}^{[k]}+\widehat{\Bh}+c\tau\Bv = \Bv+c\tau\Bv$, leading to $c(1-\tau) = 1$. Hence $c = \frac{1}{1-\tau}$.
\qquad \end{proof}
Lemma \ref{lemma-fista-acce} indicates that if FISTA and ISTA starts from the same starting point in one specific regime [B], then it will cost FISTA fewer iterations to leave this regime with larger constant step. Hence it is an acceleration compared to ISTA in regime [B]. 

\subsection{In Regimes [A] and [C]} \label{sec:regime-a-c}
On the other hand, in regimes [A] and [C], the convergence rate of the two algorithms are related to the spectral radius of $R$ and $N$.
Particularly, the rate of FISTA depends on $\tau$ and the iteration number since $\tau$ is a determined sequence based on iteration numbers.
Let $\beta$, $\gamma$ denote an eigenvalue of $R$,  $N$, respectively, and $\beta_{\max}$,    $\gamma_{\max}$ denote the
corresponding eigenvalues of largest absolute value.
As stated in Section \ref{sec:four-regimes}, we must have $1>\beta_{\max}, \gamma_{\max} \geq 0$ in regimes [A] or [C].
In addition,  by Lemma \ref{lemma:FISTA-norm}, 
$\beta$ and $\gamma$ satisfy the relation $ \gamma^2 - \gamma(1+\tau)\beta + \tau\beta = 0$. Let $\gamma_1$ and $\gamma_2$ be two roots of $\gamma$. We conclude our result in the following proposition.

\begin{proposition}\label{proposition:acceleration1}
Suppose ISTA and FISTA start from the same point in a certain regime [A] or [C] and $D^{[k]} = D^{[k+1]}$, FISTA is faster than ISTA if $1 > \beta_{\max} > \tau>0$ but slower if $0< \beta_{\max} < \tau < 1$.
If $\tau > \beta_{\max}$ then the eigenvalue $\gamma_{\max}$ of $N$ of largest absolute value is one of a pair of
complex conjugate eigenvalues.
Because $\beta_{\max}$ is a fixed value for one specific regime, with the $\tau$ growing to $1$, ISTA will be faster than FISTA toward the end.
\end{proposition}

\begin{proof}
We prove this proposition in three steps.
	
(a). From $\gamma^2 - \gamma(1+\tau)\beta + \tau\beta = 0$, $\gamma$ has real roots if $1>\beta>\frac{4\tau}{(1+\tau)^2}$ and has complex roots if $ \frac{4\tau}{(1+\tau)^2} > \beta > 0$.
	
(b). If $ \frac{4\tau}{(1+\tau)^2} > \beta > 0$, then $\gamma_1$ and $\gamma_2$ are two conjugate roots such that  $|\gamma_1|^2 = \gamma_1\gamma_2 = \tau\beta$. Note that $0<\tau<1$, we separated into two cases. If $\frac{4\tau}{(1+\tau)^2} >\beta > \tau$, then $|\gamma| = \sqrt{\tau\beta} < \beta$. If $\tau>\beta > 0$, then $|\gamma| = \sqrt{\tau\beta} > \beta$.
	
(c). If $1>\beta>\frac{4\tau}{(1+\tau)^2}$, then $\tau < \beta$, and $\gamma_1$ and $\gamma_2$ are real and $\gamma_1 = \max\{\gamma_1,\gamma_2\}= \frac{(1+\tau)\beta}{2} + \sqrt{\frac{(1+\tau)^2\beta^2 - 4\tau\beta}{4} } <\frac{(1+\tau)\beta}{2} + \sqrt{\frac{(1+\tau)^2\beta^2 - 4\beta^2}{4} }<\beta$. The first inequality is due to $\beta > \tau$ and the second one is due to $\tau < 1$.
	
Since (b) and (c) are true for any pair of $\gamma$ and $\beta$, combining them together, we get $|\gamma_{\max}| < |\beta_{\max}|$ if $\beta_{\max} > \tau$ and $|\gamma_{\max}| > |\beta_{\max}|$ if $\beta_{\max} < \tau$. In such a regime [A] or [C], as stated in Section \ref{sec:local}, both ISTA and FISTA iteration can be reduced to the power method or similar to power method on the operator $\Ra$ and $\Na$ and the rate is determined by the $|\beta_{\max}|$ and $|\gamma_{\max}|$. We complete the proof.
\qquad \end{proof}

Proposition \ref{proposition:acceleration1} concludes that if the starting points are the same in regimes [A] or [C], then ISTA will first be slower but then be faster as the iteration progresses.

\section{Examples} \label{sec:examples}
{\bf Example $1$.} We illustrate the eigen-analysis of the behavior of ISTA and FISTA and then compare them by the propositions in Section \ref{sec:accel} based on a uniform randomly generated LASSO problem. Specifically, in problem \eqref{LASSO}, $A$ and $\Bb$ are generated independently by a uniform distribution over $[-1,1]$, $A$ being $20\times 40$. Since $A$ are drawn by a continuous distribution, as noted in Lemma 4 of \cite{T13}, problem \ref{LASSO} must have a unique solution. Figure 1 shows the ISTA and FISTA convergence behavior. Using the notations from Theorems \ref{thm:linear-conv-ista} \& \ref{thm:linear-conv-fista}, the figures show the error of $\Bx$, $\|\Bx^{[k]}-\Bx^{*}\|$  (A: top curve) and the difference between two consecutive iterates of $\Bx$: $\|\Bx^{[k]}-\Bx^{[k-1]}\|$ (B: bottom curve)

Figure \ref{fig-cvg}(left) shows the behavior of ISTA.
ISTA takes 5351 iterations to converge and the flag matrix $D$ changes 25 times in total. During the first 174 iterations, the iteration passes through a few transitional phases and the flag matrix $D$ changes 20 times. After that, from iteration 175 to 483, it stays in regime [B] with an invariant $D$. Then from iteration 484 to 530, from 531 to 756, from 767 to 4722 and from 4723 to 4972, it passes through four different regimes [B]. Within each regime [B], the flag matrix $D$ is invariant. According to our analysis in Section \ref{sec:four-regimes}, there exists a Jordan chain in each of these regimes [B], indicating that we are indeed in a ``constant step" regime. In other words, the difference between two consecutive iterates $\|\Bx^{[k]}-\Bx^{[k-1]}\|$ quickly converges to $\Ra$'s eigenvector for eigenvalue $1$ in each of these regimes [B]. Taking iterations from 767 to 4972  for example, one could notice curve {\bf B} in Figure \ref{fig-cvg} (ISTA) that $\|\Bx^{[k]}-\Bx^{[k-1]}\|$ is a constant from iteration $767$ to $4722$. Finally, at iteration 4973, it reaches and stays in the final regime [A], converging linearly in 378 steps. Indeed, the iterates are close enough to the final optimum so that the flags never change again. The linear convergence rate depends on the spectral radius of $R$, i.e. upper left part of $\Ra$, which is $\rho(R) = 0.9817$, separated from the $\Ra$'s largest eigenvalue $1$,  consistent with Theorem \ref{thm:linear-conv-ista}. 

Figure \ref{fig-cvg}(right) shows the behavior of FISTA.
FISTA takes 1017 iterations to converge and the flag matrix $D$ changes 42 times in total. After flag matrix $D$ changes 42 times in initial 258 iterations, it reaches the final regime [A] at iteration 259 and converges linearly in 758 steps. Since $\Na$ varies at each iteration due to varying $\tau$, the convergence rate changes very slightly step by step. The spectral radius of $N$, i.e. upper left part of the operator $\Na$ in the last step is $\rho(N) = 0.9914$. Actually, the largest eigenvalues of $N$ are a complex conjugate pair, $0.9843 \pm 0.1185i$. Note that they are complex numbers  because of the increasing $\tau$, as stated in Proposition \ref{proposition:acceleration1} in Section \ref{sec:accel}. Hence, in the final regime, the convergence to eigenvector for eigenvalue $1$ of $\Na$ will oscillate in the invariant subspaces spanned by the eigenvectors of two conjugate complex pairs. This explains why curves in Figure \ref{fig-cvg}(FISTA) oscillates in the latter part of the FISTA convergence.

\begin{figure}
	\centering
	\includegraphics[width =6.4cm]{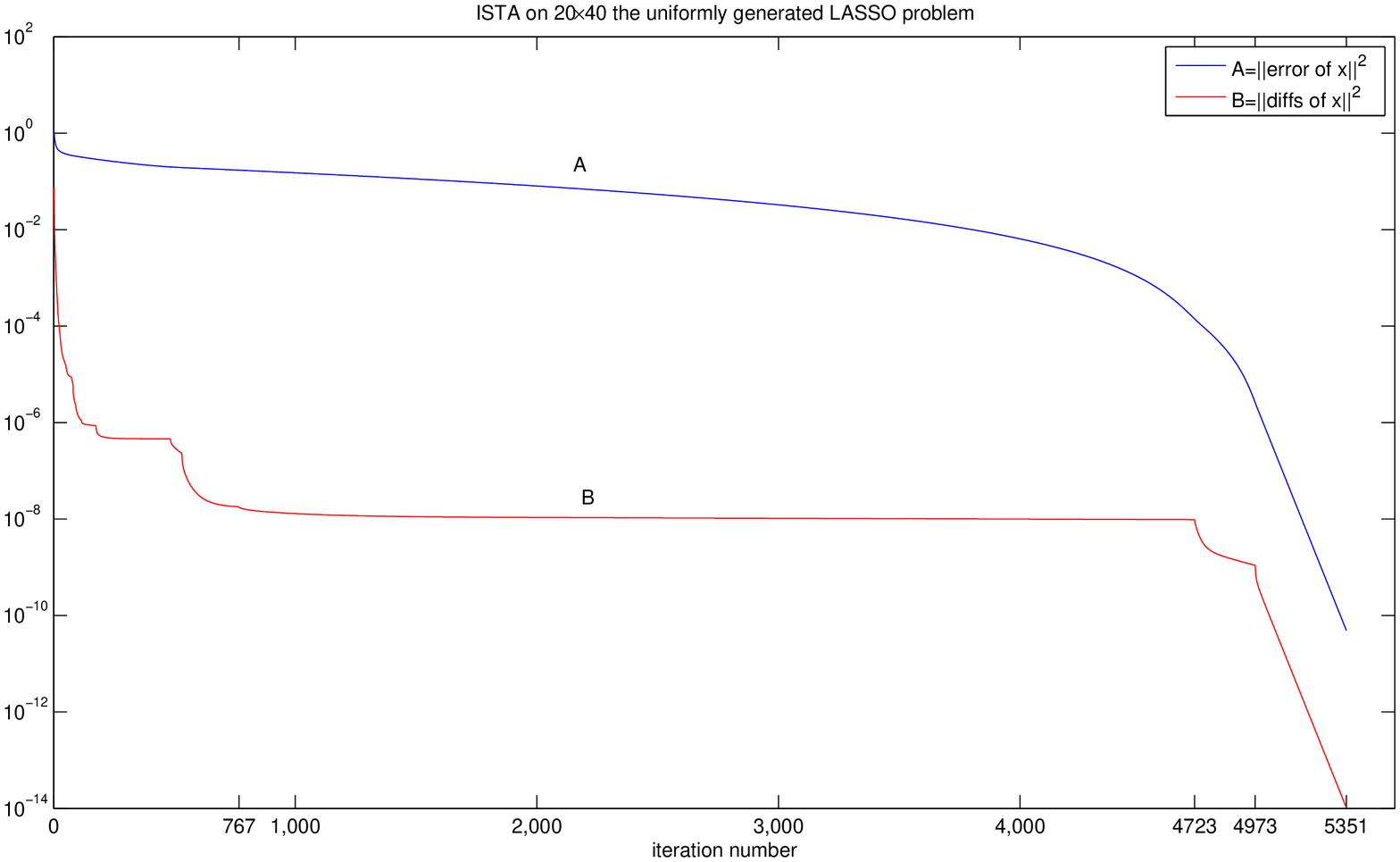} \label{eps-ista}
	\includegraphics[width =6.4cm]{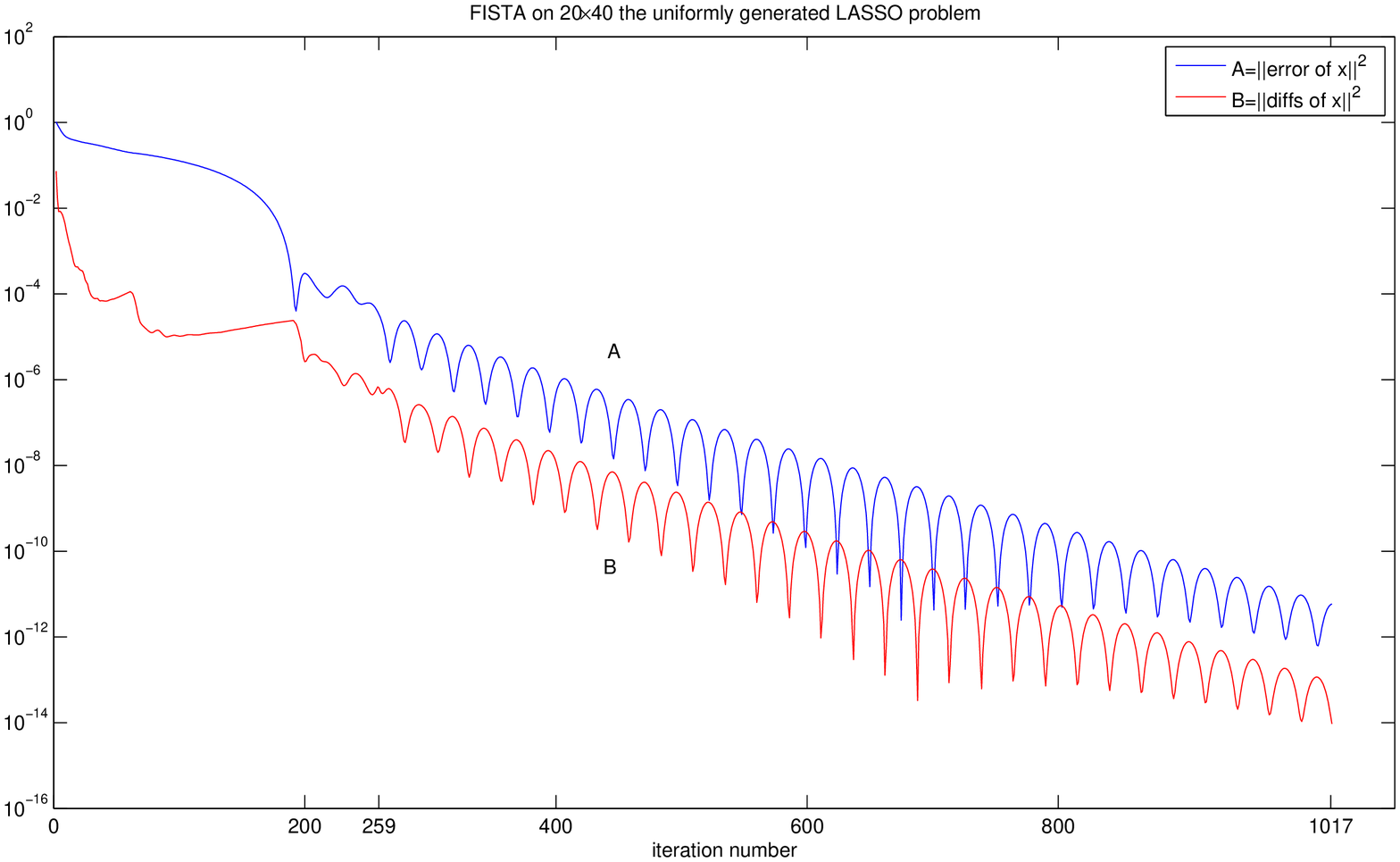} \label{eps-fista}
	\caption{ISTA (left) and FISTA (right) on Example 1: Curves {\bf A}:  $\|\Bx^{[k]}-\Bx^{*}\|$. {\bf B}:  $\|\Bx^{[k]}-\Bx^{[k-1]}\|$.  }
\label{fig-cvg}
\end{figure}

Figure \ref{fig:eigs} shows the eigenvalues of the operators $\Ra$ and $\Na$ during the final regime. One notices that the eigenvalues for the $\Ra$ from (\ref{ISTA-Raug}) lie strictly on the interval $(0,1)$ and  eigenvalues for $\Na$ lie
close to the boundary but strictly inside the circle $\Dcal(\rfrc{1}{2}, \rfrc{1}{2})$ (except for $0$ and $1$),
consistent with Lemmas \ref{lemma:ISTA-norm} $\&$ \ref{lemma:FISTA-norm}.

\begin{figure}
	\centering
	\includegraphics[width=6.4cm]{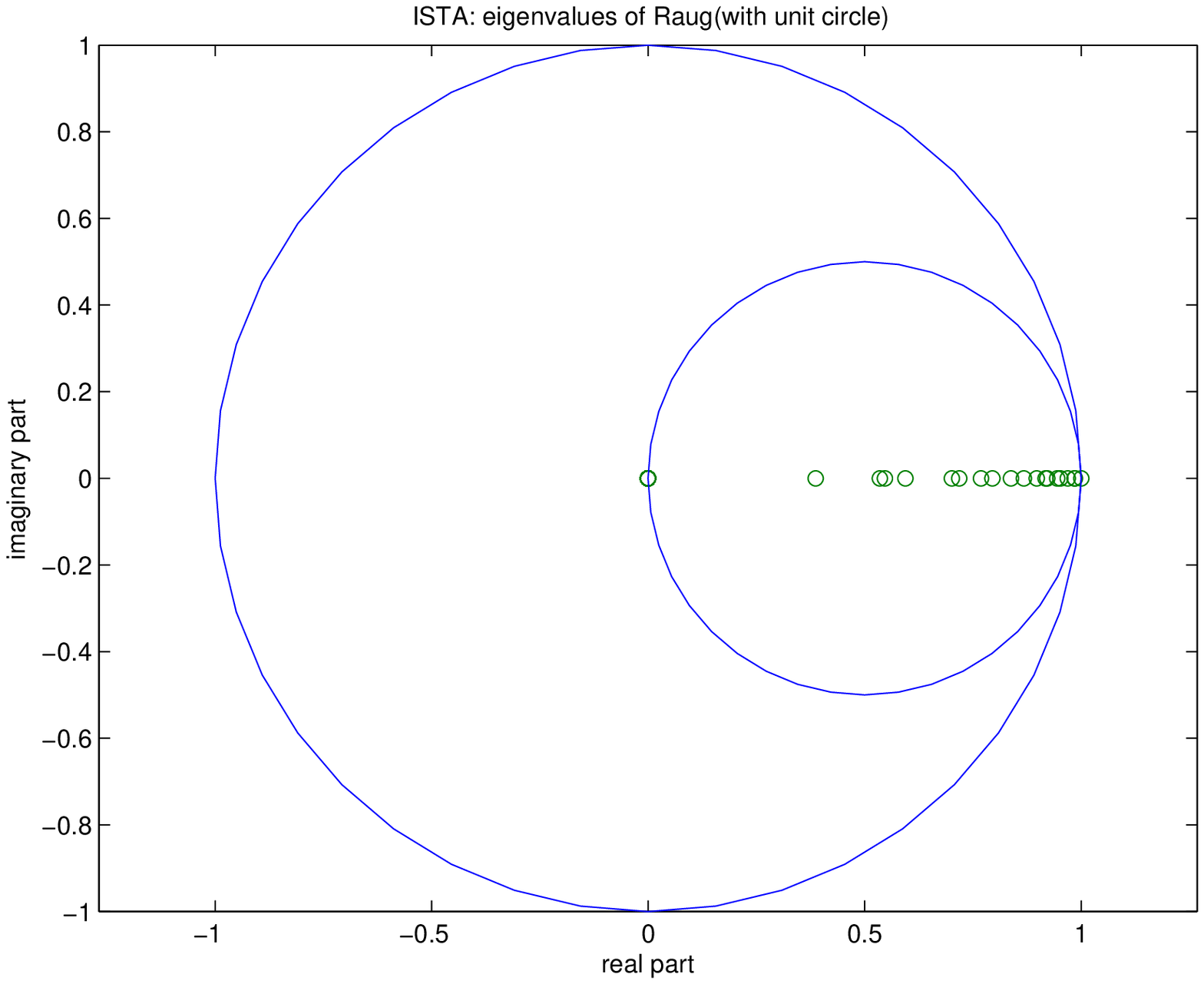}
	\includegraphics[width=6.4cm]{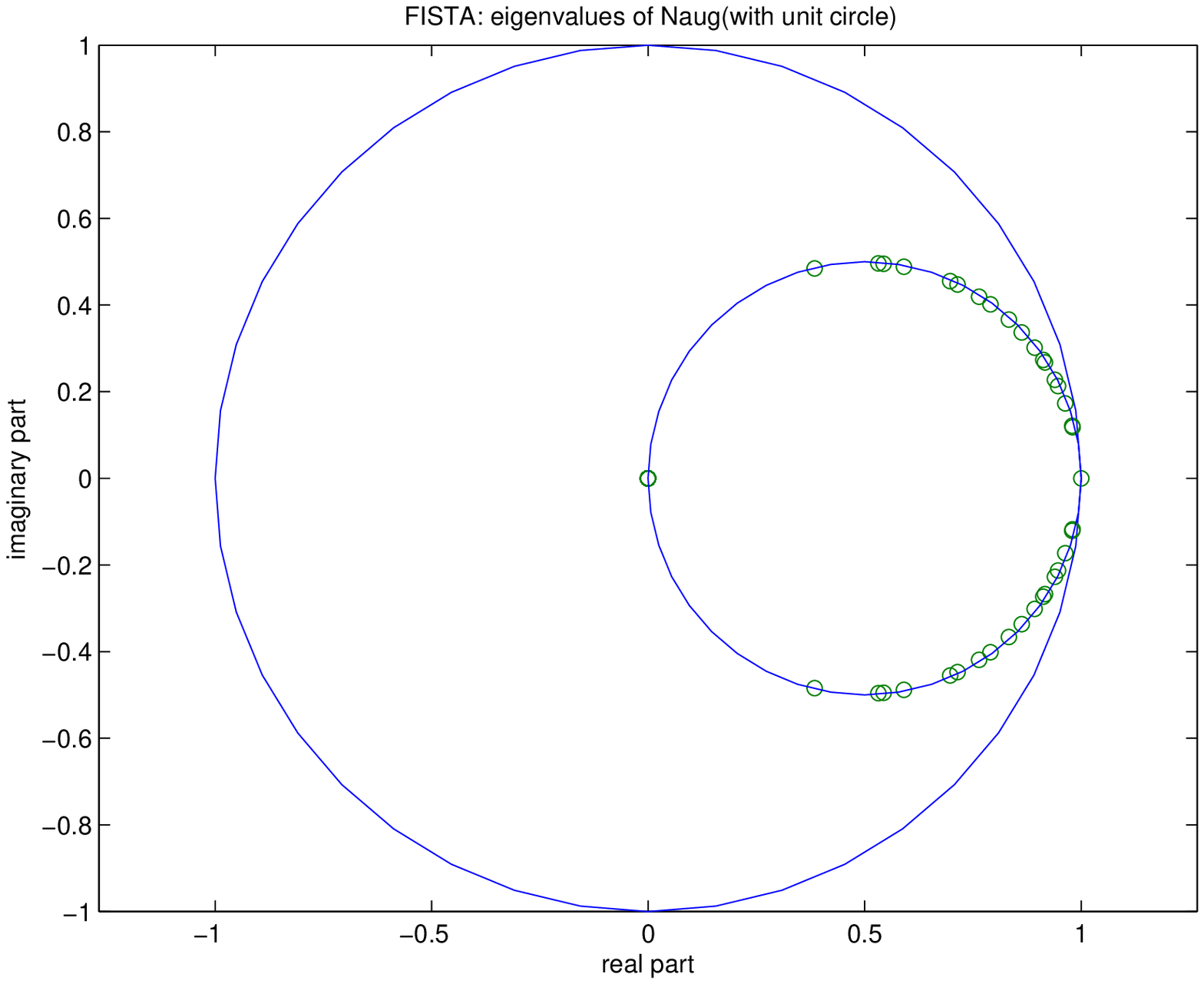}
	\caption{ISTA (left) and FISTA (right) on Example 1: Eigenvalues of ISTA operator $\Ra$ and FISTA operator $\Na$ on the complex plane during the last regime of the iteration process.
	The unit circle and \(\Dcal(\frc{1}{2},\frc{1}{2})\) are shown for reference.
	}
\label{fig:eigs}
\end{figure}

Comparing ISTA with FISTA, we make two remarks based on our propositions in Section \ref{sec:accel}:

a. It costs FISTA many fewer steps (259 iterations) than ISTA (4973 iterations) to get to the final regime. The main reason is that FISTA has much larger constant steps in regime [B] so that it can jump out of that regime more quickly. Though this will lead to more changes of regimes (flag matrix $D$ changes 42 times, 17 more times than ISTA), the overall iteration numbers have been cut down, consistent with Lemma \ref{lemma-fista-acce} in Section \ref{sec:accel}. One can also notice this in Figure \ref{fig-cvg}(FISTA) that difference of iterates will not remain at a constant number for a long time and the iterates will be soon in the final regime.

{ b. At iteration 259 of FISTA, $\tau=0.9886$ while $\rho(R) = 0.9817$. Proposition \ref{proposition:acceleration1} predicts  $\rho(N) > \rho(R) $ for the rest of the iterations. Indeed, $\rho(N) = 0.9914 >\rho(R)$. This means ISTA is faster than FISTA in each of their final regime. In other words, if one detects the arrival of final regime and then change from FISTA to ISTA at step $259$, the algorithm (denoted as F/ISTA in Figure \ref{fig:ifista}) should have converge faster than the standard FISTA. As shown in Figure \ref{fig:ifista}, the algorithm of this idea converges only in $696$ iterations with the same accuracy compared to $1017$ iterations of FISTA. }
\begin{figure}
	\centering
	\includegraphics[width =8cm]{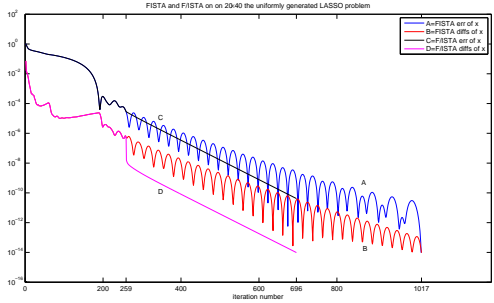}
	\caption{Curve C and D show hybrid F/ISTA on Example 1, i.e. follow FISTA during initial 259 iterations to reach the final regime and then switch to ISTA until it converges. The total iteration number is 696, far less than 1017 iterations of FISTA alone with the same accuracy. The FISTA on Example 1 (curve A and B) is shown for reference.}
\label{fig:ifista}
\end{figure}

{\bf Example 2} We consider an example of compressed sensing. The focus of this example is not on the efficiency comparison among different methods but to show its local behavior to support our analysis. Suppose there exists a true sparse signal represented by a $n$-th dimension vector $\Bx$ with $k$ non-zero elements. We observe the image of $\Bx_s$ under the linear transformation $A\Bx_s$, where $A$ is the so-called measurement matrix. Our observation thus should be 
\be
\Bb = A\Bx_s + \epsilon
\ee
where $\epsilon$ is the observation noise. The goal is to recover the sparse vector $\Bx_s$ from the measurement matrix $A$ and observation $\Bb$.
For this example, we let $A\in \R^{m\times n}$ be Gaussian matrix whose elements are \textit{i.i.d} distributed as $\mathcal{N}(0,1)$ with $m = 128$ and $n = 1024$, $\epsilon$ be a vector whose elements are \textit{i.i.d} distributed as $\mathcal{N}(0,\sigma^2)$ with $\sigma = 10^{-3}$.
The original true signal for the problem is generated by choosing the locations of $\Bx$'s $k(=10)$ nonzeros uniformly at random, then setting those locations to values drawn from $\mathcal{N}(0, 2^2)$. 

We solve this compressd sensing problem by model \eqref{LASSO} with both ISTA and FISTA method. For ISTA, we set $\lambda = 1$ and the final recovered signal $\widehat{\Bx}_\text{opt}$ , i.e. the optimal solution of model \eqref{LASSO} under $\lambda = 1$, has the relative error $\| \widehat{\Bx}_\text{opt} - \Bx_s \| / \|\Bx_s\|= 5.34\times 10^{-3}$. As for FISTA, we set $\lambda = 0.5$ and the final recovered signal $\widetilde{\Bx}_\text{opt}$, i.e. the optimal solution of model \eqref{LASSO} under $\lambda = 0.5$, has the relative error $\| \widetilde{\Bx}_\text{opt} - \Bx_s \| / \|\Bx_s\|= 2.65\times 10^{-3}$.

It costs ISTA $2822$ iterations to reach the final regime and finally converges in totally $3001$ iterations. On the other hand, it costs FISTA $372$ iterations to reach the final regime and converges in totally $717$ iterations. Figure 3 show their convergence behavior. It can be seen that curves of the difference of iterates in both two figures remain at a constant number for many iterations. This is because they are in the constant regimes such that the difference between consecutive iterates (curves B) are converging to a constant vector. But such iteration number for FISTA obviously is shorter than ISTA because it has a larger constant step size, as we indicated in Section \ref{sec:accel}. 

Finally, both algorithms has linear convergence for the final regimes. The linear convergence rate for ISTA is the second largest eigenvalue of $\Ra$, which equal to $ 0.9587$. The linear convergence rate  for FISTA from step $372$ to the last step $717$ is the second largest eigenvalue of $\Nak$, which remains at $ 0.9752$ for $k$ from $372$ to $717$, slower than ISTA rate. From this example, by
the time FISTA reaches the final regime, $\tau$ is so close to 1 that the rate for FISTA is
changing very little. At iteration 372, $\tau^{[372]}= 0.9920$, which is already greater than $0.9587$,
predicting that switching to ISTA at this point would be advantageous.

\begin{figure}
	\centering
	\includegraphics[width =6.4cm]{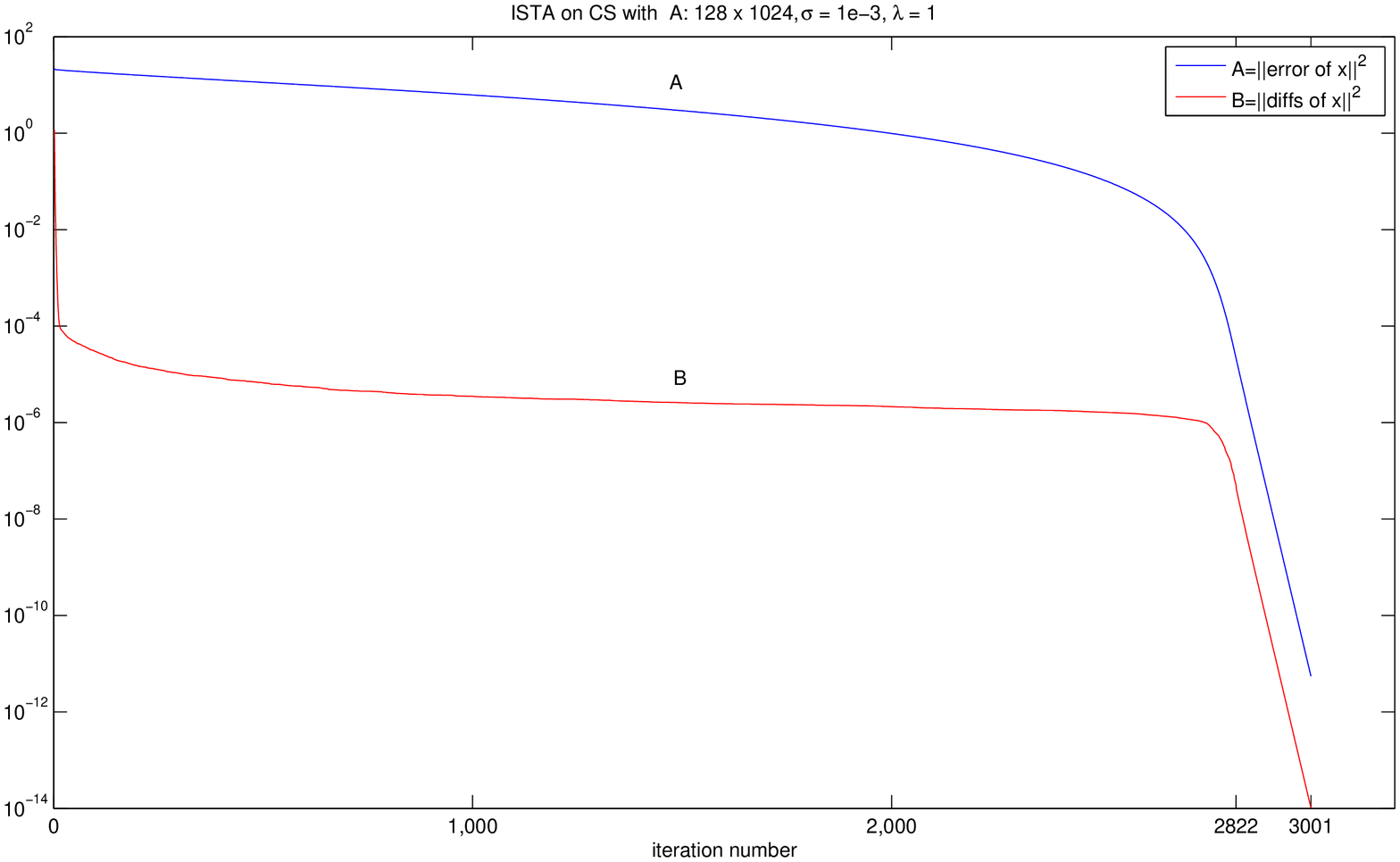}
	\includegraphics[width =6.4cm]{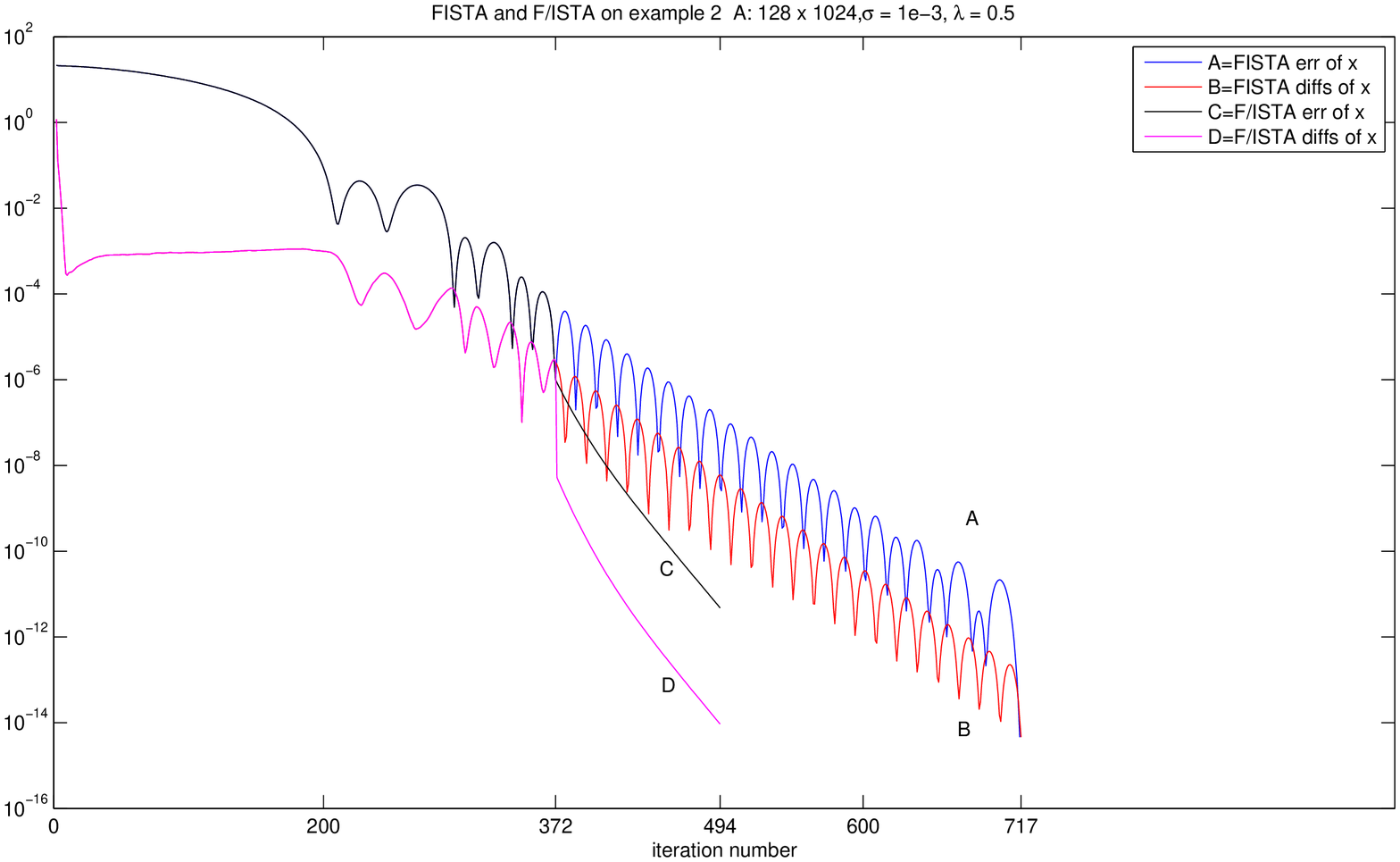}
	\caption{Left: ISTA on Example 2. Right:  Curve A and B is the FISTA on Example 2. Curve C and D is the local convergence behavior of F/ISTA on Example 2, i.e. first run FISTA during initial 372 iterations to reach the final regime and then run ISTA until it converges. The total iteration number is 494, far less than 717 iterations of FISTA with the same accuracy. }
\label{fig:stacs}
\end{figure}

{ Particularly, if one detects the arrival of final regime and then changes from FISTA to ISTA, the algorithm (denoted as F/ISTA in Figure \ref{fig:stacs}) will have a faster linear convergence rate. Much computational cost will be reduced. As shown in Figure \ref{fig:stacs}, the algorithm of this idea converges only in $494$ iterations with the same accuracy compared to $717$ iterations of FISTA. }

\section{Conclusion} \label{sec:end}
In this paper, we  show the locally linear convergence of ISTA and FISTA, applied to the LASSO problem. Extending the same techniques as in \cite{B13}, both algorithms can be modeled as the matrix recurrence form and thus the spectrum can be used to analyze their convergence behaviors. It is shown that the method normally passes through several regimes of four types and eventually settles on a `linear regime' in which the iterates converge linearly with the rate depending on the absolute value of the second largest eigenvalue of the matrix recurrence.

In addition, we provide a way to analyze every type of the regime. Such analysis in terms of regimes allows one to study the aspect of acceleration of FISTA. It is well known that FISTA is faster than ISTA according the worst case complexity bound. Our analysis gives another way to show how both methods behave during the whole iterations.  It turns out that FISTA is not always faster than ISTA in regime [A] and [C], depending on the continually growing stepsize. But in general FISTA is faster because of its acceleration in regime [B].

\bibliography{fistaPaper}
\bibliographystyle{siam.bst}

\end{document}